\documentclass[11pt,a4paper,reqno]{amsart}

\input amssymb.sty
\usepackage{amsmath}
\usepackage{amsfonts}
\usepackage{epsfig}

\newcommand{\CP}{\mathbb{CP}}
\newcommand{\C}{\mathbb{C}}
\newcommand{\Z}{\mathbb{Z}}
\newcommand{\G}{\Gamma}
\newcommand{\p}{\pi_1}

\newcommand{\vp}{\varphi}
\newcommand{\ri}{\rightarrow}

\def\a{\alpha}

\def\Dl{\Delta}
\def\dl{\delta}
\def\g{\gamma}
\def\lm{\lambda}

\def\vp{\varphi}
\def\vt{\vartheta}

\def\p{\pi}
\def\po{\pi_1}

\def\wc{\widetilde{C}}
\def\wb{\widetilde{B}}
\def\wz{\widetilde{Z}}

\def\bG{\bar{\Gamma}}

\def\uz{\underline{Z}}
\def\bz{\bar{Z}}

\def\Gmo{\Gamma^{-1}}
\def\Gmt{\Gamma^{-2}}

\def\ztw{Z^2}
\def\zmtw{Z^{-2}}
\def\zmo{Z^{-1}}
\def\zt{Z^{(3)}}
\def\wzt{\widetilde{Z}^2}
\def\bzt{\bar{Z}^{(3)}}
\def\bztw{\bar{Z}^2}
\def\uz{\underline{Z}}

\def\ss1{\mbox{\sf 1}}

\def\ztw{Z^2}
\def\zmtw{Z^{-2}}
\def\zmo{Z^{-1}}
\def\zt{Z^{(3)}}
\def\wzt{\widetilde{Z}^2}
\def\bzt{\bar{Z}^{(3)}}
\def\bztw{\bar{Z}^2}
\def\uz{\underline{Z}}
\def\wG{\widetilde{\Gamma}}
\def\Gmo{\Gamma^{-1}}
\def\Gmt{\Gamma^{-2}}

\newcommand{\usr}{\underset\sim\rightarrow}
\begin{document}
\newtheorem{corollary}{Corollary}[section]
\newtheorem{lemma}{Lemma}[section]
\newtheorem{prs}{Proposition}[section]
\newtheorem{remark}{Remark}[section]

\newtheorem{thm}{Theorem} [section]
\title{On non Fundamental Group Equivalent Surfaces}
\author[M. Friedman]{Michael Friedman}

\address{
Michael Friedman, Department of Mathematics, Bar-Ilan University,
52900 Ramat Gan, Israel} \email{fridmam@macs.biu.ac.il}

\author[M. Teicher]{Mina Teicher$^1$}
\stepcounter{footnote} \footnotetext{This work is partially
supported by the Emmy Noether Research Institute for Mathematics
and the Minerva Foundation of Germany and the Israel Science
Foundation grant \# 8008/02-3 (Excellency Center ``Group Theoretic
Methods in the Study of Algebraic Varieties").}
 \address{Mina Teicher, Department of Mathematics,
 Bar-Ilan University, 52900 Ramat Gan, Israel}
 \email{teicher@macs.biu.ac.il}

\keywords{fundamental group, generic projection, curves and
singularities, branch curve.} \subjclass[2000]{14F35, 14H20,
14H30, 14J28, 14Q05, 14Q10, 20F36, 57M12}

\maketitle

\begin{abstract}

In this paper we present an example of two polarized K3 surfaces
which are not Fundamental Group Equivalent (their fundamental
groups of the complement of the branch curves are not isomorphic;
denoted by FGE) but the fundamental groups of their related Galois
covers are isomorphic. For each surface, we consider a generic
projection to $\CP^2$ and a degenerations of the surface into a
union of planes - the ``pillow" degeneration for the non-prime
surface and the ``magician" degeneration for the prime surface. We
compute the Braid Monodromy Factorization (BMF) of the branch
curve of each projected surface, using the related degenerations.
By these factorizations, we compute the above fundamental groups.
It is known that the two surfaces are not in the same component of
the Hilbert scheme of linearly embedded K3 surfaces. Here we prove
that furthermore they are not FGE equivalent, and thus they are
not of the same Braid Monodromy Type (BMT) (which implies that
they are not a projective deformation of each other).

\end{abstract}

\section[Introduction]{Introduction} \label{intro}
Given $X\subset \CP^n$ a smooth algebraic surface of degree $m$,
one can obtain information on $X$ by considering it as a branched
cover of $\CP^2$. It is well--known that for $X \longrightarrow
\CP^2$ a generic projection, the branch locus is a plane curve
$\bar{S} \subset \CP^2$ which is, in general, singular, and its
singularities are nodes and cusps. Let $S \subset \C^2 \subset
\CP^2$ be a generic affine portion of $\bar{S}$.

 It was proven in \cite{KuTe} that if the Braid Monodromy Factorizations (BMF) of the branch loci
of two surfaces $X_1$ and $X_2$ are Hurwitz-equivalent, then the
surfaces are diffeomorphic. Moreover, if the factorizations are
not Hurwitz-equivalent, then $X_1$ and $X_2$ are not projectively
deformation equivalent. Therefore, the BMT invariant (the
equivalence class of a BMF) is really in the ``middle", i.e.,
between the diffeomorphism equivalence and the projectively
deformation equivalence. We need to find an algorithm that decides
whether two BMFs are equivalent. In general, it was shown in
\cite{LT} that there is no finite algorithm which determines
whether two positive factorizations are Hurwitz- equivalent.
However, \cite{LT} did not examine the particular case of the
BMFs. Therefore, we have to extract the information contained in
the braid monodromy factorization via the introduction of more
manageable (but less powerful) invariants.

Two discrete invariants are induced from the BMF of the branch
curve -- $S$: the fundamental group of the complement of the
branch curve (see \cite{AFT},\cite{FT},\cite{Mo},\cite{MoTe5}) and
its subquotient: the fundamental group of the Galois Cover of $X$
(see \cite{L},\cite{MoRoTe},\cite{MoTe0}). We say that two
surfaces are Fundamental Group Equivalent (FGE) if their
fundamental groups of the complement of the branch curve are
isomorphic.

In this article we present two surfaces, which are embeddings of a
K3 surface with respect to two different linear systems; therefore
they are diffeomorphic. Due to the nature of the particular linear
systems, these embedded surfaces are not projectively deformation
equivalent. It is also known that any two K3 surfaces can be
abstractly deformed one into the other. Thus one can raise the
questions: Are the surfaces FGE? Are the fundamental groups of the
corresponding Galois covers isomorphic? Here we prove that
although the latter groups are isomorphic, the surfaces are not
FGE. Therefore, these surfaces are also not BMT--equivalent, which
means that the surfaces are not in the same component of the
Hilbert scheme of linearly embedded K3 surfaces.

\textbf{Acknowledgments}: We are very grateful to Ciro Ciliberto,
for pointing out the degeneration used for the ``magician" K3
surface and for other helpful discussions. We also wish to thank
Meirav Amram for her fruitful remarks.

\section[Preliminaries]{Preliminaries: The K3 surfaces and the BMT
invariant} \label{sec2}

In this section we recall the main definitions and constructions
regarding the two embeddings of the $K3$ surface, and the braid
monodromy factorization (=BMF) related to a (branch) curve. We
begin with the introduction of the two embeddings of a $K3$
surface.

\subsection[Two embeddings of a $K3$ surface]{Two embeddings of a $K3$
surface} \label{sec2_1}

Recall that the surfaces with Kodaira dimension which equals to 0
,that are simply connected, have in fact trivial canonical bundle,
and are called $K3$ surfaces. The invariants for such surfaces are
$p_g=1$, $q=0$, $e = 24$. The moduli space of all $K3$ surfaces is
$20$-dimensional.

Most $K3$ surfaces are not algebraic; the algebraic ones are
classified by an infinite collection (depending on an integer $g
\geq 2$) of $19$-dimensional moduli spaces.  The general member of
the family has a rank one Picard group, generated by an ample
class $H$ with $H^2 = 2g-2$; the general member of the linear
system $|H|$ is a smooth curve of genus $g$, and this linear
system maps the $K3$ surface to $\mathbb{P}^g$ as a surface of
degree $2g-2$. For example, a $K3$ surface is a smooth quartic
surface in $\mathbb{P}^3$. The quartic surfaces in $\mathbb{P}^3$
form the family with $g=3$. The integer $g$ is called the
\emph{genus} of the family.

The first embedded surface is a $K3$ surface of genus 9, embedded
in $\CP^9$ by the pillow (2,2)-pillow degeneration (see \cite{CMT}
for details). The resulting embedding can be degenerated into a
union of 16 planes, such that the whole degenerated object would
``resemble a pillow" (see figure 1 for clarification). We denote
by $X_1$ the embedded $K3$ surface, and by $(X_1)_0$ the
degenerated surface (see \cite{MoTe5} for an explicit definition
of a degeneration).

\begin{center}
\epsfig{file=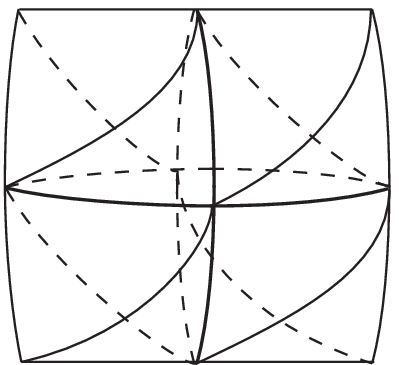}\\
\small{Figure 1: $(X_1)_0$ -- the (2,2)-pillow degeneration:\\
every triangle denotes a plane}
\end{center}
\bigskip
\begin{center}
\epsfig{file=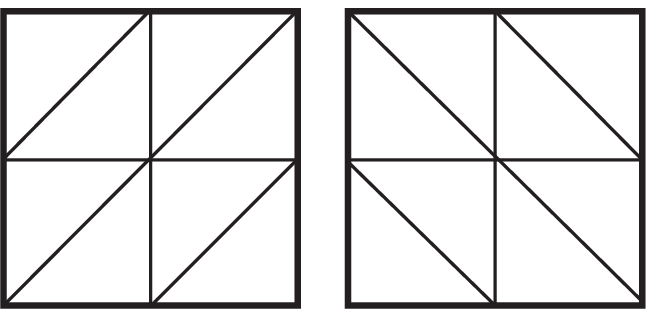}

\small{Figure 2: A 2-dimensional figure of $(X_1)_0$: the
boundaries are identified \\(top to top, bottom to bottom, side to
side)}
\end{center}

The degeneration process has a ``local inverse" -- the
regeneration process (see an explanation in the following
subsection), and for it we need to fix a numeration of vertices
(and the lines; see \cite{ATCM} for details). This is done in the
following way (see figure 3):

\begin{center}
\epsfig{file=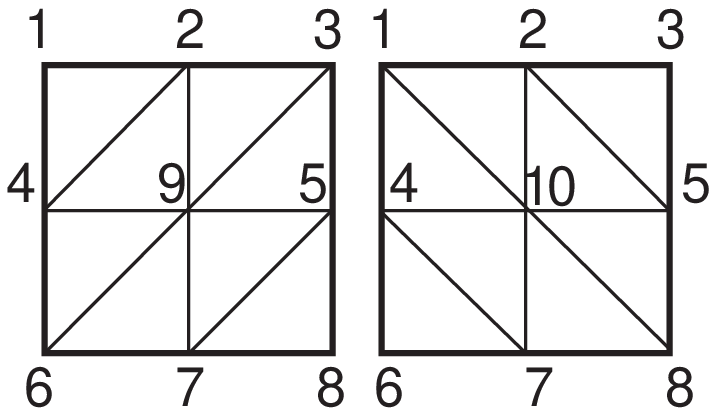}\\
\small{Figure 3: The numeration of the singular points of
$(X_1)_0$}
\end{center}

The 16 planes meet each other along a total of 24 lines, each
joining 2 of the 10 coordinate points. We numerate  the lines as
follows: if $L$ has endpoints $a<b$ and $M$ has endpoints $c<d$,
then $L<M$ if $b<d$ or $b=d$ and $a<c$. This gives a total
ordering of the lines, which we interpret as a numbering from 1 to
24, as shown in figure 4.
\begin{center}
\epsfig{file=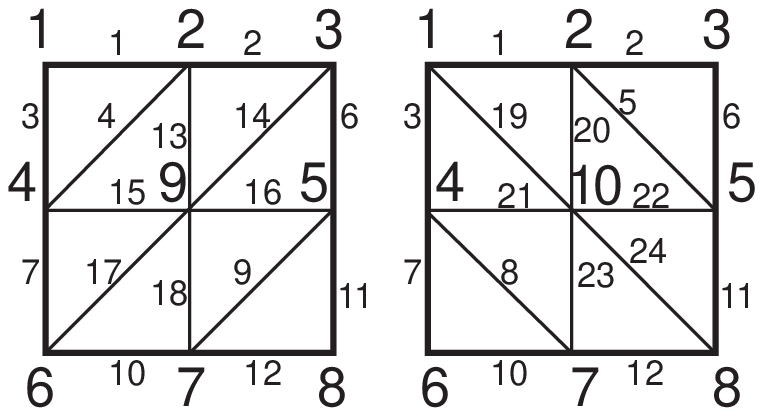}\\

\small{Figure 4: The numeration of the intersection lines of
$(X_1)_0$}
\end{center}
Under a general projection $\p:(X_1)_0 \ri \CP^2$, each of the 16
planes is mapped isomorphically to $\CP^2$. The ramification locus
$R_1$ of $\p$ is a local isomorphism. Here $R_1$ is exactly the 24
lines. Let $(S_1)_0 = \p(R_1)$ be the degenerated branch curve. It
is a line arrangement, composed of the image of the 24 lines.

The second embedded surface is also an embedded $K3$ surface of
genus 9 in $\CP^9$. We call this surface the ``magician" surface,
since its degeneration ``resembles" a magician's hat. The surface
and its degeneration into a union of 16 planes are described in
\cite{CM}. The dual graph of the degenerated surface is presented
explicitly in \cite[pg. 430]{CM} - and from it we can build the
degenerated surface (see figure 5).
\begin{center}
\epsfig{file=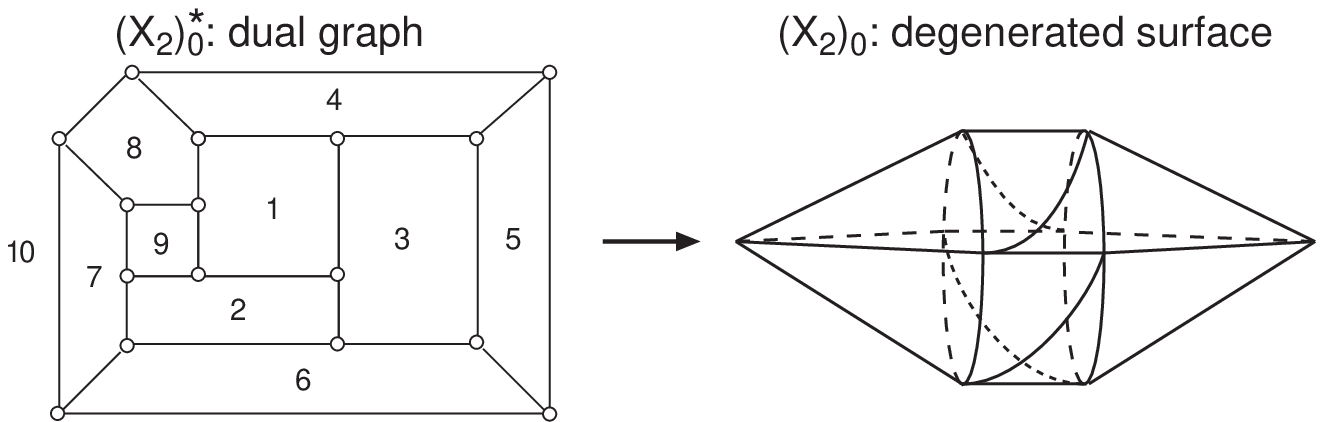}\\

\small{Figure 5: Every point in the dual graph represents a plane;\\
every plane represents a point}
\end{center}
Denote by $X_2$ this embedded surface, and by $(X_2)_0$ the
degenerated surface. We can depict a 2-dimensional graph of
$(X_2)_0$, where the boundaries are identified (see figure 6):

\begin{center}
\epsfig{file=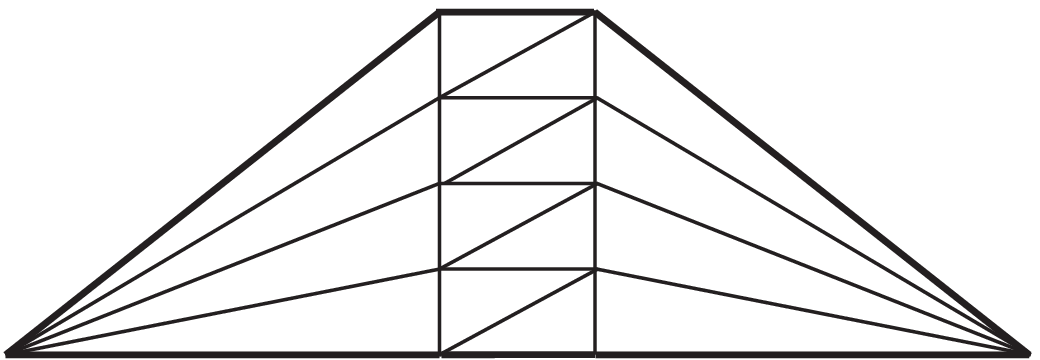}\\

\small{Figure 6: $(X_2)_0$ - the boundaries are identified (top to
bottom)}
\end{center}

Once again, we numerate the vertices and then the edges. We note
that the extreme edges of the graph $(X_2)_0$ are actually
4--points: singular points in the degenerated surface which are
the intersection of four planes. In order to regenerate it (see
\cite{Robb} for the possible degenerations of this point), we need
to numerate the vertices in such a way that the number of
``entering" and ``exiting" lines from these points will be equal.
Therefore, we numerate them as vertices 5 and 6. Following the
symmetry appearing in the graph, we numerate the other vertices as
follows (see figure 7):

\begin{center}
\epsfig{file=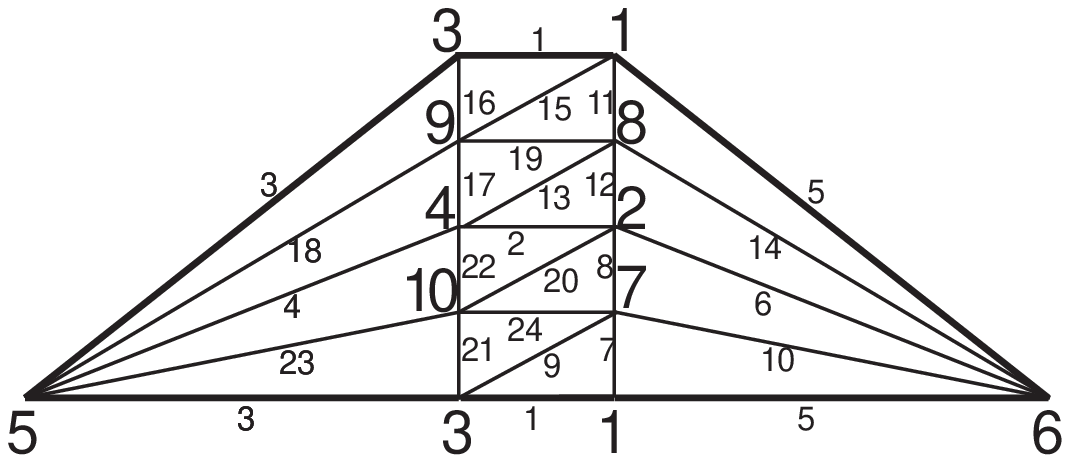}\\

\small{Figure 7: Numeration of $(X_2)_0$}
\end{center}
Note that $(X_2)_0$ also contains 24 intersection lines and 10
singular points. We denote by $(S_2)_0 = \pi_2(R_2)$ the
degenerated branch curve with respect to a generic projection
$\pi_2:(X_2)_0 \ri \CP^2$.

Since every two $K3$ surfaces are diffeomorphic, $X_1$ and $X_2$
are also diffeomorphic. Note that the Hilbert scheme of embedded
linearly normal K3 surfaces can be reducible. This is indeed the
case here -- the Picard group Pic$X_1$ is generated by
$\frac{1}{2}H$ (where $H$ is the hyperplane class; see \cite{CMT})
and Pic$X_2$ is generated by $H$ (see \cite{CM}).

Two polarized K3 surfaces are projectively deformation equivalent
if and only if there is a diffeomorphism which carries the
hyperplane class to the hyperplane class. As indicated above, this
is not the case. We show in the following sections that these
surfaces are also not BMT--equivalent, and that the fundamental
groups of complement of the branch curve can also be used in order
to differentiate between irreducible components of the Hilbert
scheme. Thus it is a topological invariant that arises in
algebro-geometric considerations.

\subsection[The braid group and the BMF]{The braid group and the BMF}

Recall that computing the braid monodromy is the main tool to
compute fundamental groups of complements of curves. The reader
who is familiar with this subject can skip the following
definitions. We begin by defining the braid monodromy associated
to a curve.

Let $D$ be a closed disk in $ \mathbb{R}^2,$ \ $K\subset Int(D),$
$K$ finite, $n= \#K$. Recall that the braid group $B_n[D,K]$ can
be defined as the group of all equivalent diffeomorphisms $\beta$
of $D$ such that $\beta(K) = K\,,\, \beta |_{\partial D} =
\text{Id}\left|_{\partial D}\right.$. \\

\noindent \textbf{Definition}:\ \underbar{$H(\sigma)$, half-twist
defined by $\sigma$}

Let $a,b\in K,$ and let $\sigma$ be a smooth simple path in
$Int(D)$ connecting $a$ with $b$ \ s.t. $\sigma\cap K=\{a,b\}.$
Choose a small regular neighborhood $U$ of $\sigma$ contained in
$Int(D),$ s.t. $U\cap K=\{a,b\}$. Denote by $H(\sigma)$ the
diffeomorphism of $D$ which switches $a$ and $b$ by a
counterclockwise $180^\circ$ rotation and is the identity on
$D\setminus U$\,. Thus it defines an element of $B_n[D,K],$ called
{\it the half-twist defined by
$\sigma$ }.\\

Denote $[A,B] = ABA^{-1}B^{-1},\,\langle A,B\rangle =
ABAB^{-1}A^{-1}B^{-1}$. We recall the Artin presentation of the
braid group:
\begin{thm} \label{thm2_1} $B_n$ is generated by the half-twists $H_i$ of a frame
{$H_i$} and all the relations between $H_1,...,H_{n-1}$ follow
from:\begin{center} $[H_i,H_j] = 1\,\,$ if\,\,\,$
|i-j|>1$\\$\langle H_i,H_j\rangle = 1 \,\,\text{if}\,\,
\,\,|i-j|=1$.
\end{center}
\end{thm}

Assume that all of the points of $K$ are on the $X$-axis (when
considering $D$ in $\mathbb{R}^2$). In this situation, if $a,b \in
K$, and $z_{a,b}$ is a path that connects them, then we denote it
by $Z_{a,b} = H(z_{a,b})$. If $z_{a,b}$ is a path that goes below
the $X$-axis, then we denote it by $\underline Z_{a,b}$, or just
$Z_{a,b}$. If $z_{a,b}$ is a path that goes above the $x$-axis,
then we denote it by $\overline Z_{a,b}$. We also denote by
$\overset{(c-d)}{\underline{Z}_{a,b}}$
($\underset{(c-d)}{\bar{Z}_{a,b}}$) the braid induced from a path
connecting the points $a$ and $b$ below (resp. above) the
$X$-axis, going above (resp. below) it from the point $c$ till
point $d$.\\

\noindent \textbf{Definition}: \underline{The braid monodromy
w.r.t. $S,\pi,u$}

Let $S$ be a curve, $S\subseteq \C^2$ . Let $\pi: S\to\C^1$ be
defined by $\pi(x,y)=x.$ We denote $\deg\pi$ by $m.$ Let
$N=\{x\in\C^1\bigm| \#\pi^{-1}(x)< m\}.$
   Take $u\notin N,$ s.t.  $\Re(x)\ll u$ \ $\forall x\in N.$
Let  $ \C^1_u=\{(u,y)\}.$  There is a  naturally defined
homomorphism
$$\pi_1(\C^1-N,u)\xrightarrow{\vp} B_m[\C_u^1,\C_u^1\cap S]$$ which
is called {\it the braid monodromy w.r.t.} $S,\pi,u,$ where $B_m$
is the braid group. We sometimes denote $\vp$ by $\vp_u$. In fact,
denoting by $E$, a big disk in $\C^1$ s.t. $E \supset N$, we can
also take the path in $E\setminus N$ not to be a loop, but just a
non-self-intersecting path. This induces a diffeomorphism between
the models $(D,K)$ at the two ends of the considered path, where
$D$ is a big disk in $\C^1_u$, and
$K = \C_u^1\cap S \subset D$.\\

\noindent \textbf{Definition}:  $\underline{\psi_T, \
\text{Lefschetz diffeomorphism induced by a path} \ T }$

Let  $T$ be a path in $E\setminus N$ connecting $x_0$ with $x_1$,
$T:[0, 1]\ri E\setminus N$. There exists a continuous family of
diffeomorphisms $\psi_{(t)}: D\ri D,\ t\in[0,1],$ such that
$\psi_{(0)}=Id$, $\psi_{(t)}(K(x_0))=K(T(t)) $ for all
$t\in[0,1]$, and  $\psi_{(t)}(y)= y$ for all $y\in \partial D$.
For emphasis we write $\psi_{(t)}:(D,K(x_0))\ri(D,K(T(t))$. A
Lefschetz diffeomorphism induced by a path $T$ is the
diffeomorphism
$$\psi_T= \psi_{(1)}: (D,K(x_0))\usr (D,K(x_1)).$$
Since $ \psi_{(t)} \left( K(x_{0})\right) = K(T(t))$ for all $t\in
[0,1]$, we have a family of canonical isomorphisms
$$\psi_{(t)}^{\nu}: B_p\left[ D, K(x_{0})\right] \usr B_p\left[
D, K({T(t)})\right], \ \quad \text{for all} \, \, t\in[0,1].$$\\

We recall Artin's theorem on the presentation of the Dehn twist of
the braid group as a product of braid monodromy elements of a
geometric-base (a base of $\p = \p(\C^1 - N, u)$ with certain
properties; see \cite{MoTe1} for definitions).\\
\begin{thm} Let $S$ be a curve transversal to the line in infinity, and
$\vp$ is a braid monodromy of $S , \vp:\p \rightarrow B_m$. Let
{$\delta_i$} be a geometric (free) base (g-base) of $\p,$ and $
\Delta^2$ is the generator of Center($B_m$). Then:
$$\Delta^2 = \prod\vp(\delta_i).$$ This product is also defined as
the \textsl{braid monodromy factorization} (BMF) related to a
curve $S$.\end{thm}

Note that if $x_1,...,x_{n-1}$ are the generators of $B_n$, then
we know that $\Delta^2 = (x_1\cdot\ldots\cdot x_{n-1})^n$ and thus
deg($\Delta^2$) = $n(n-1)$.

So in order to find out what is the braid monodromy factorization
of $\Delta_p^2$, we have to find out what are $\vp
(\delta_i),\,\forall i$. We refer the reader to the definition of
a \textit{skeleton} (see \cite{MoTe2}) $\lambda_{x_j}, x_j \in N$,
which is a model of a set of paths connecting points in the fiber,
s.t. all those points coincide when approaching
$A_j=$($x_j,y_j$)$\in S$, when we approach this point from the
right. To describe this situation in greater detail, for $x_j \in
N$, let $x_j' = x_j + \alpha$. So the skeleton in $x_j$ is defined
as a system of paths connecting the points in $K(x_j') \cap
D(A_j,\varepsilon)$ when $0 < \alpha \ll \varepsilon \ll 1$,
$D(A_j,\varepsilon)$ is a disk centered
in $A_j$ with radius $\varepsilon$.\\

For a given skeleton, we denote by
$\Delta\langle\lambda_{x_j}\rangle$ the braid which rotates by
$180^\circ$
 counterclockwise a small neighborhood of the given
skeleton. Note that if $\lambda_{x_j}$ is a single path, then
$\Delta\langle\lambda_{x_j}\rangle = H(\lambda_{x_j})$.

We also refer the reader to the definition of $\delta_{x_0}$, for
$x_0 \in N$ (see \cite{MoTe2}), which describes the Lefschetz
diffeomorphism induced by a path going below $x_0$, for different
types of singular points (tangent, node, branch; for example, when
going below a node, a half-twist of the skeleton occurs and when
going below a tangent point, a full-twist occurs).

We define, for $x_0 \in N$, the following number:
$\varepsilon_{x_0} = 1,2,4$ when ($x_0, y_0$) is a branch / node /
tangent point (respectively). So we have the following statement
(see \cite[Prop. 1.5]{MoTe2}):

Let $\gamma_j$ be a path below the real line from $x_j$ to $u$,
s.t. $\ell(\gamma_j)=\delta_j$. So
$$\vp_u(\delta_j) = \vp(\delta_j) =
\Delta \bigg\langle
(\lambda_{x_j})\bigg(\prod\limits_{m=j-1}^{1}\delta_{x_m}\bigg)
\bigg \rangle ^{\varepsilon_{x_j}}.$$ When denoting $\xi_{x_j} =
(\lambda_{x_j})\bigg(\prod\limits_{m=j-1}^{1}\delta_{x_m}\bigg)$
we get --
$$\vp(\delta_j)
= \Delta\langle(\xi_{x_j})\rangle^{\varepsilon_{x_j}}.$$ Note that
the last formula gives an algorithm to compute the needed
factorization.

For a detailed explanation of the braid monodromy, see \cite{MoTe1}.\\

We shall now define an equivalence relation on the BMF.\\
{\textbf{Definition}:\, \underbar{\emph{Hurwitz moves}}}:

Let $\vec t= (t_1,\ldots ,t_m)\in G^m$\,. We say that $\vec s
=(s_1,\ldots ,s_m)\in G^m$ is obtained from $\vec t$ by the
Hurwitz move $R_k$ (or $\vec t$ is obtained from $\vec s$ by the
Hurwitz move $R^{-1}_k$) if
$$
s_i = t_i \quad\text{for}\  i\ne k\,,\, k+1\,,\\
s_k = t_kt_{k+1}t^{-1}_k\,,\\
s_{k+1} =t_k\,.
$$
\textbf{Definition}:\, \underbar{\emph{Hurwitz move on a
factorization}}

Let $G$ be a group $t\in G.$  Let  $t=t_1\cdot\ldots\cdot t_m=
s_1\cdot\ldots\cdot s_m$ be two factorized expressions of $t.$ We
say that $s_1\cdot\ldots\cdot s_m$ is obtained from
$t_1\cdot\ldots\cdot t_m$ by a Hurwitz move $R_k$ if $(s_1,\ldots
,s_m)$ is obtained from $(t_1,\ldots ,t_m)$ by a Hurwitz move
$R_k$\,.\newpage \textbf{Definition}:\, \underbar{\emph{Hurwitz
equivalence of factorization}}

Two factorizations are Hurwitz equivalent if they are obtained
from each other by a finite sequence of Hurwitz moves.\\
\textbf{Definition}:\, \underbar{\emph{Braid monodromy type of
curves (BMT)}}

Two curves $S_1$ and $S_2$ are of the same BMT (denoted by
$\cong$) if they have related BMF's that are equivalent.
\bigskip

In 1998, the following theorem was proved (\cite{KuTe}) :

\begin{thm}
If $S_1 \cong S_2$, then $S_1$ is isotopic to $S_2$ (when $S_1,
S_2$ are any curves).
\end{thm}

Thus, an invariant of surfaces can be derived from the BMT of the
branch curve of a surface.\\\\
\textbf{Definition}:\, \underbar{\emph{Braid monodromy type of
surfaces (BMT)}}

  The BMT of a projective surface is the BMT of the
branch curve of a generic projection of the surface embedded in a
projective space by means of a complete linear system.

\bigskip
Consequently, the following was proved (\cite{KuTe}):
\begin{thm}
The BMT of a projective surface $X$ determines the diffeomorphism
type of $X$.
\end{thm}

We recall now the regeneration methods.

The regeneration methods are actually, locally, the reverse
process of the degeneration method. When regenerating a singular
configuration consisting of lines and conics, the final stage in
the regeneration process involves doubling each line, so that each
point of $K$ corresponding to a line labelled $i$ is replaced by a
pair of points, labelled $i$ and $i'$. The purpose of the
regeneration rules is to explain how the braid monodromy behaves
when lines are doubled in this manner. We denote by $Z_{i,j} =
H(z_{i,j})$ where $z_{i,j}$ is a path connecting points in $K$.

The rules are (see \cite[pg. 336-337]{MoTe4}):
\begin{enumerate}
\item \textbf{First regeneration rule}: The regeneration of a
branch point
of any conic:\\
A factor of the braid monodromy of the form $Z_{i,j}$ is replaced
in the regeneration by $Z_{i',j}\cdot
\overset{(j)}{\underline{Z}}_{i,j'}$\medskip
\item \textbf{Second regeneration rule}: The regeneration of a node:\\
A factor of the form $Z^2_{ij}$ is replaced by a factorized
expression $Z^2_{ii',j} := Z^2_{i'j}\cdot Z^2_{ij}$ ,\\
$Z^2_{i,jj'} := Z^2_{ij'}\cdot Z^2_{ij}$ or by $Z^2_{ii',jj'} :=
Z^2_{i'j'}\cdot Z^2_{ij'}Z^2_{i'j}\cdot Z^2_{ij}$. \medskip \item
\textbf{Third regeneration rule}: The regeneration of a tangent
point:\\
A factor of the form $Z^4_{ij}$ in the braid monodromy factorized
expression is replaced by\\ $Z^3_{i,jj'} :=
(Z^3_{ij})^{Z_{jj'}}\cdot (Z^3_{ij}) \cdot
(Z^3_{ij})^{Z^{-1}_{jj'}}$.
\end{enumerate}
As a result, we get a factorized expression, which, by
\cite{KuTe}, determines the diffeomorphism type of our surface,
and, by \cite{VK}, determines $\pi_1(\CP^2 -\overline S)$. This is
explained in the following paragraphs.

Assume that we have a curve $\bar{S}$ in $\CP^2$ and its BMF. Then
we can calculate the groups\\ $\pi_1(\CP^2 -\overline S)$ and
$\pi_1(\C^2 - S)$ (where  $S = \bar{S} \cap \C^2$).

Recall that a $g$-base is an ordered free base of $\p(D \backslash
F,v)$, where $D$ is a closed disc, $F$ is a finite set in
Int($D$), $v \in \partial D$ which satisfies several conditions;
see \cite{MoTe1}, \cite{MoTe2} for the explicit definition.

Let $\{\G_i\}$ be a $g$-base of $G = \pi_1(\C_u-S,u),$ where $\C_u
= \C \times u$, and here $S = \C_u \cap S$. We cite now the
Zariski-Van Kampen Theorem (for cuspidal curves) in order to
compute the relations between the generators in $G.$

\begin{thm} \label{thm2_5}{\rm Zariski-Van Kampen (cuspidal curves version)} Let
$\overline S$ be a cuspidal curve in $\CP^2$. Let
$S=\C^2\cap\overline S.$ Let $\vp$ be a braid monodromy
factorization w.r.t. $S$ and $u.$ Let $\vp=\prod\limits_{j=1}^p
V_j^{\nu_j},$ where $V_j$ is a half-twist and $\nu_j=1,2,3.$

For every $j=1\dots p$, let $A_j,B_j\in\pi_1(\C_u-S,u)$ be such
that $A_j,B_j$ can be extended to a $g$-base of $\pi_1(\C_u-S,u)$
and $(A_j)V_j=B_j.$ Let $\{\G_i\}$ be a $g$-base of
$\pi_1(\C_u-S,u)$ corresponding to the $\{A_i, B_i \}$, where
$A_i, B_i$ are expressed in terms of $\G_i$. Then
$\pi_1(\C^2-S,u)$ is generated by the images of $\{\G_i\}$ in
$\pi_1(\C^2-S,u)$ and the only relations are those implied from
$\{V_j^{\nu_j}\},$ as follows:
$$\begin{cases} A_j\cdot B_j^{-1}&\quad\text{if}\quad \nu_j=1\\
[A_j,B_j]=1&\quad\text{if}\quad \nu_j=2\\
\langle A_j,B_j\rangle=1&\quad\text{if}\quad \nu_j=3.\end{cases}$$
$\pi_1(\CP^2-\overline S,*)$ is generated by $\{\G_i\}$ with the
above relations and one more relation $\prod\limits_i \G_i=1.$
\end{thm}

The following figure illustrates how to find $A_i, B_i$ from the
half-twist $V_i = H(\sigma)$:

\begin{center}
\epsfig{file=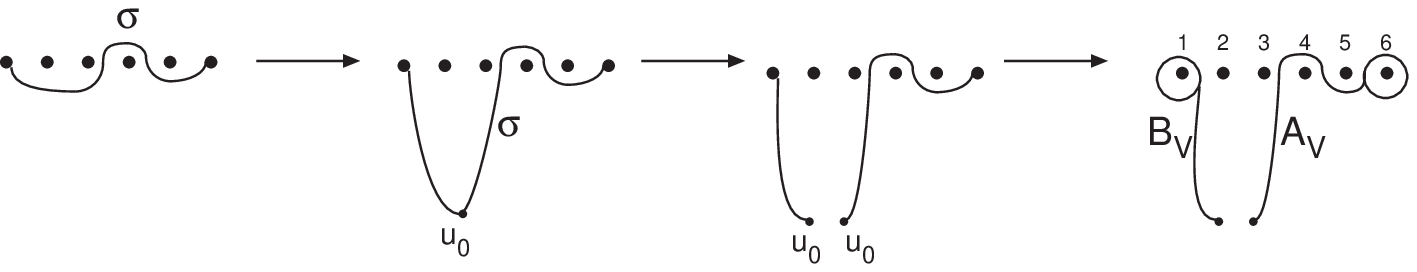}\\
\small{Figure 8}
\end{center}

So: $$A_V = \G_4^{-1}\G_6\G_4,\,B_V = \G_1.$$

We finish this subsection by recalling the definition of
$\tilde{B}_n$.

\textbf{Definition}:\, 1. Let $X,Y$ be two half-twists in $B_n =
B_n(D,K)$. We say that $X,Y$ are \emph{transversal} if they are
defined by two simple paths $\xi, \eta$ which intersect
transversally in one point  different from their ends.\\
2. Let $N$ be the normal subgroup of $B_n$ generated by conjugates
of $[X,Y]$, where $X,Y$ is a transversal pair of half-twists.
Define
$\tilde{B}_n = B_n/N$.\\
3. Let $Y_i,\,i=1,..,4$ be four half-twists in $B_n$ (resp.
$\tilde{B}_n$) corresponding to simple paths $\eta_1,...,\eta_4$.
Assume that $\eta_i,\,i=1,...,4$, could be chosen so that they
form a quadrangle without self intersections and such that in its
interior there are no points of $K$. Then we say that
$Y_1,Y_2,Y_3,Y_4$ form a \emph{good quadrangle} in $B_n$ (resp. in
$\tilde{B}_n$).

\begin{lemma} \label{lem2_1}
If $y_1,y_2,y_3,y_4 \in \tilde{B}_n$ form a good quadrangle then
$y_1^2y_3^2 = y_2^2y_4^2$.
\end{lemma}
\begin{proof} See \cite[section 1.1]{Mo}.\end{proof}

\section[Computing the BMFs]{Computing the BMFs} \label{sec3}

Let $\vp_1, \vp_2$ be the BMF of the branch curve of the first
(resp. second) K3 surface.  Before computing $\vp_1, \vp_2$, we
need a few notations.  Denote the intersection lines on $(X_i)_0$
as $\{ \hat{L}_{i,j} \}^{24}_{j=1},\linebreak i = 1,2$ (recall
that $(X_i)_0$ is the degeneration of the K3-surfaces $X_i, i =
1,2)$, and by $\{ \hat{v}_{i,j} \}^{10}_{j=1},\linebreak i = 1,2$
the intersection points of these lines.  Take generic projections
$\p_i : (X_i)_0 \rightarrow \mathbb{C} \mathbb{P}^2$, and let
$(S_i)_0$ be the branch curve in $\mathbb{C} \mathbb{P}^2,
(\vp_i)_0$ - their braid monodromy, and $L_{i,j} =
\pi_i(\hat{L}_{i,j}), i = 1,2, \ j, ..., 24$. So, $(S_i)_0 =
\bigcup\limits^{24}_{j=1} L_{i,j}, v_{i,j} \doteq \p_i
(\widetilde{v}_{i,j} ), i = 1,2, \ j = 1, ..., 24$ are the
singular points of $(S_{i})_0$.  Let $C_i$ be the union of all
lines connecting pairs of the $v_{i,j} \in (S_i)_0$.  $(S_i)_0$ is
a subcurve of $C_i$.  By  \cite[Theorem IX]{MoTe1}, we get a full
description of the braid monodromy of $C_i: \Dl^2_{C_1} =
\Pi^{1}_{j=10} C_{i,j} \Dl^2_{v_{i,j}} (i = 1,2)$ with an
appropriate description of L.V.C.  We use this formula to obtain a
description of $(\vp_i)_0$ by deleting factors that involve lines
which do not appear in $(S_i)_0$.  Thus, we get $(\vp_i)_0 =
\Dl^2_{(S_i)_0} = \Pi^{1}_{j=10} \wc_{i,j}
\widetilde{\Dl}^2_{v_{i,j}}$. We describe each factor separately.

\noindent \underline{$\wc_{i,j}$} : The factors $\wc_{i,j}$
correspond to parasitic intersections; these are intersections
created by lines that do not intersect in $\mathbb{C}
\mathbb{P}^9$ but may intersect in $\mathbb{C} \mathbb{P}^2$.  By
\cite{MoTe1} we know that $\wc_{i,j} = \prod\limits_{v_{i,j} \in
L_{i,t}} \ D_{i,t}$, where $D_{i,t} =
\prod\limits_{\stackrel{p<t}{L_{i,p} \cap L_{i,t}} = \emptyset}
\widetilde{Z}^2_{pt}$. For $i = 1$, the global BMF, together with
the $\wc_{1,j}$ is presented in\\ \cite[Section 4.1]{ATCM}. For
$i=2$, we have (by \cite[Thm. X.2.1]{MoTe1}):

\def\pl{\prod\limits}

$$D_{2,1}  =  id \quad D_{2,2}  =  \uz^2_{1,2} \quad  D_{2,3} =
\uz^2_{2,3} \quad D_{2,4} = \overset{(2)}{\uz^2_{1,4}} $$
$$D_{2,5}  =  \bz^2_{2,5} \bz^2_{3,5} \uz^2_{4,5} \quad D_{2,6}
=\overset{(2-3)}{\uz^2_{1,6}}\overset{(4)}{\uz^2_{3,6}}
\uz^2_{4,6} \quad D_{2,7}   =  \pl_{i=2,3,4,6} \bz^2_{i,7} \quad$$
$$D_{2,8} = \pl_{i=1,3,4,5} \underset{(7)}{\bz^2_{i,8}} \quad
D_{2,9}  =  \underset{(7-8)}{\bz^2_{2,9}} \
\overset{(5-6)}{\uz^2_{4,9}} \ \overset{(6)}{\uz^2_{5,9}} \
\uz^2_{6,9} \quad D_{2,10}
=\pl_{i=1,2,3,4}\underset{(7-9)}{\bz^2_{i,10}}   $$
$$D_{2,11}  =
\pl_{\stackrel{i=2,3,4,}{\scriptscriptstyle{6,8,9,10}}}
\bz^2_{i,11} \quad D_{2,12}  =
\pl_{\stackrel{i=1,3,4,5}{\scriptscriptstyle{7,9,10}}}
\underset{(11)}{\bz^2_{i,12}} \quad D_{2,13} =
\pl_{\stackrel{i=1,3,5,}{\scriptscriptstyle{7,\cdots ,10}}}
\underset{(11-12)}{\bz^2_{i,13}}$$

$$D_{2,14}  =  \pl_{\stackrel{i=1,\cdots ,
4}{\scriptscriptstyle{7,8,9}}} \underset{(11-13)}{\bz^2_{i,14}}
\quad D_{2,15} =
\pl_{\stackrel{i=2,\cdots,14}{\scriptscriptstyle{1\neq5,7,11}}}
\bz^2_{i,15} \quad D_{2,16}  =  \pl_{\stackrel{i=2,\cdots ,
14}{\scriptscriptstyle{i \neq 3,9}}}
\underset{(15)}{\bz^2_{i,16}}$$
$$ D_{2,17} =
\pl_{\stackrel{i=1,\cdots,14}{\scriptscriptstyle{i\neq 2,4,13}}}
\underset{(15-16)}{\bz^2_{i,17}} \quad D_{2,18}  =
\pl_{\stackrel{i=1,\cdots , 14}{\scriptscriptstyle{i \neq 3,4}}}
\underset{(15-17)}{\bz^2_{i,18}} \quad D_{2,19} =
\pl_{i=1,\cdots,10} \underset{(15-18)}{\bz^2_{i,19}}$$
$$ D_{2,20}  =  \pl_{\stackrel{i=2,\cdots ,
19}{\scriptscriptstyle{i \neq 2,6,8,12}}} \bz^2_{i,20} \quad
D_{2,21} = \pl_{\stackrel{i=2,\cdots,19}{\scriptscriptstyle{i\neq
3,9,16}}} \underset{(20)}{\bz^2_{i,21}} \quad D_{2,22}  =
\pl_{\stackrel{i=1,\cdots , 19}{\scriptscriptstyle{i \neq
2,4,13,14}}} \underset{(20-21)}{\bz^2_{i,22}} $$
$$ D_{2,23} = \pl_{\stackrel{i=1,\cdots,19}{\scriptscriptstyle{i\neq
3,4,18}}}
\underset{(20-22)}{\bz^2_{i,23}} \quad D_{2,24}  =
\pl_{\stackrel{i=1,\cdots , 19}{\scriptscriptstyle{i \neq
7,\cdots,10}}} \underset{(20-23)}{\bz^2_{i,24}} $$and

$$ \wc_{2,1}  =
\pl_{\stackrel{t=5,7}{\scriptscriptstyle{11,15}}} D_{2,t} \quad
\wc_{2,2} = \pl_{\stackrel{t=2,6,8}{\scriptscriptstyle{12,20}}}
D_{2,t} \quad \wc_{2,3} =
\pl_{\stackrel{t=3,9}{\scriptscriptstyle{16,21}}} D_{2,t}$$
$$\wc_{2,4} = \pl_{\stackrel{t=4,13}{\scriptscriptstyle{17,22}}} D_{2,t}
\wc_{2,5}  =  \pl_{t=18,23} D_{2,t}  \quad \wc_{2,6} =
\pl_{t=10,14} D_{2,t}$$ $$ \wc_{2,7}  =  D_{2, 24} \quad
\wc_{2,18} = D_{2,19} \quad  \wc_{2,9} = \wc_{2,10} = i d$$

Recall that a point in a totally degenerated surface is called a
$k$-point if it is a singular point which is the intersection of
$k$ planes.

\noindent {$\widetilde{\underline \Dl}^2_{v_{i,j}}$} :  In
$(S_1)_0$, we have six points, which are 6-point $(v_{1,j} , j =
2,4,5,7,9,10)$ and four points which are 3-point $(v_{1,j} , j =
1,3,6,8$ ; note that the regeneration of this 3-point is not
similar to the regular 3-point.  See \cite{ATCM} for the braid
monodromy factorization of the regeneration  of our 3-point).

In $(S_2)_0$, we have eight points which are  5-point $(v_{2,j} ,
1 \leq j \leq 10, j \neq 5,6)$  and two points which are 4-point
$(v_{2,j} , j = 5,6$).  Note that the original branch curve,
$S_2$, has also a few extra branch points. The existence of the
extra branch points will be proved later (see Proposition
(\ref{prs3_5})).

The local braid monodromies, which are
$\widetilde{\Dl}^2_{v_{2,j}}$, are introduced and regenerated in
the following paragraphs.  We denote the outcoming local BMF,
resulting from the total regeneration
$\widetilde{\Dl}^2_{v_{2,j}}$, as $\vp_{2,j}$.  Thus after
performing a total regeneration to the whole BMF, the resulting
BMF will be of the form $\vp_2 = \pl^{1}_{i=10} C_i \vp_i \pl
b_i$, where $b_i$ are braids corresponding to the extra branch
points.

Before presenting the expressions for local and global BMFs, we
give a few  notations.  Let $a,b,c,d \in \mathbb{Z}$; denote:\\ \\
$ F_u(a,b,c,d) : = F_u(F_u)_{Z^{-1}_{a,a'} Z^{-1}_{d,d'}}, $ where
$\{b,c\} < \{a,d\}$, and $c<b, a<d$ and
$$
F_u = Z^{(3)}_{bb',a} Z^2_{a'd}
(Z^2_{ad})_{Z^2_{bb'a}}(Z^3_{bb',d})_{Z^2_{bb',a}}
(Z_{cb'})_{Z^2_{b'd}Z^2_{b'a}}
(Z_{c'b})_{Z^2_{bd}Z^2_{ba}Z^2_{bb'}}$$\\ $ F_m(a,b,c,d): = F_m
\cdot (F_m)_{Z^{-1}_{a,a'}Z^{-1}_{d,d'}}$ where $a < \{b,c\} < d $
and
$$
F_m = Z^{(3)}_{a',cc'} \cdot  Z^{(3)}_{bb',d} \cdot \wz_{c,b'}
\cdot \wz_{b',c} (\uz^2_{a',d})_{Z^2_{c',cc'}} \cdot \uz^2_{ad}
$$
where
$$
\wz_{cb'} = (\uz_{cb'})_{Z^2_{b',d}Z^2_{cc'}Z^2_{a'c}} \quad
\mbox{and} \quad \wz_{b',c} = (\uz_{b',c})_{Z^2_{b'd}Z^2_{a'c'}}
$$\\
$ F_\ell(a,b,c,d) : = F_\ell \cdot
(F_\ell)_{Z^{-1}_{a,a'}Z^{-1}_{d,d'}} $ where $\{b,c\}>\{a,b\}$
and
$$F_\ell  =  Z^2_{a'd} \cdot
Z^{(3)}_{d',cc'} (Z^2_{a'd'})_{Z^2_{d',cc'}Z^2_{a'd}}
(Z^{(3)}_{a',cc'})_{Z^2_{a'd}} \cdot
(\bz_{cb'})_{Z^2_{d',c}Z^2_{a'c}Z^2_{a'd}}
(Z_{cb'})_{Z^2_{cc'}Z^2_{d'c'}Z^2_{a'c'}Z^2_{a'd}}.
$$

Note that for $(\vp_1)_0$ and the singular points of $(S_1)_0$,
the regeneration process was already done \cite{ATCM}, and thus we
have the following:
\\
\begin{thm} \label{thm3_1}The BMF of the branch curve of $X_1$ is
$$
\vp_1 = \pl^{1}_{j=10} C_{1,j} \vp_{1,j}
$$
where $C_{1,j},\, \vp_{1,j}$ can be found in \cite{ATCM}.

\end{thm}

\begin{proof} see \cite{ATCM}.\end{proof}

Thus, we have to compute the BMF of the branch curve of $X_2$. We
begin by citing the results about the points $v_{2,5}$ and
$v_{2,6}$; these are 4-points and for this type, the BMF of a
fully regenerated neighbourhood was computed in \cite{AT}.

\begin{prs} \label{prs3_1}
The local braid monodromy $\vp_{2,5}$ in a small neighbourhood
around $v_{2,5}$ has the following form:
$$
\vp_{2,5} = F_u(18,4,3,23)
$$
and the local braid monodromy $\vp_{2,6}$ (for $v_{2,6}$) has the
same form, when substituting $3 \rightarrow 5,\\ \ 4 \rightarrow
6, \ 18 \rightarrow 10, \ 14 \rightarrow 23 $.

\end{prs}

\begin{proof} See \cite{AT}.\end{proof}

We now move on to compute the local braid monodromy around a small
neighbourhood of $v_{2,3}$, which is a 5-point.  We will give -- for this
point -- a
detailed treatment for the computation of the local BMF, while for
the other points $(v_{2,j} , j = 1,2,4,7, \cdots , 10)$ we will
just give the final results.

We examine the point $v_{2,3}$ in the degenerated surface
$(X_2)_0$.  Drawing a local neighbourhood of $v_{2,3}$ and
numerating the lines $-\,L_i(1 \leq i \leq 5)$ locally, we get:

\begin{center}
\epsfig{file=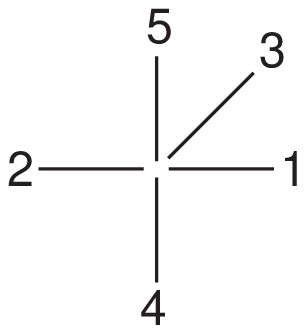}\\

\small{Figure 9}
\end{center}

By the degeneration process, line 3 is regenerated first. By the
Claim in \cite[Section 2]{MoTe4}, we know that line $L_3$ is
regenerated into a conic.  More explicitly, we get that after
regenerating $V = \bigcup\limits^5_{i=1} L_i$ in a small
neighbourhood $U$ of $v_{2,3}$, $L_3$ turns into a conic $Q_3$
such that $Q_3$ is tangent to $L_1$ and $L_5$.  Denote the
resulting branch curve, after the regeneration by $\widetilde{V}$.
Thus, the singularities of $T = \widetilde{V} \cap U$ are as in
the figure below:

\begin{center}
\epsfig{file=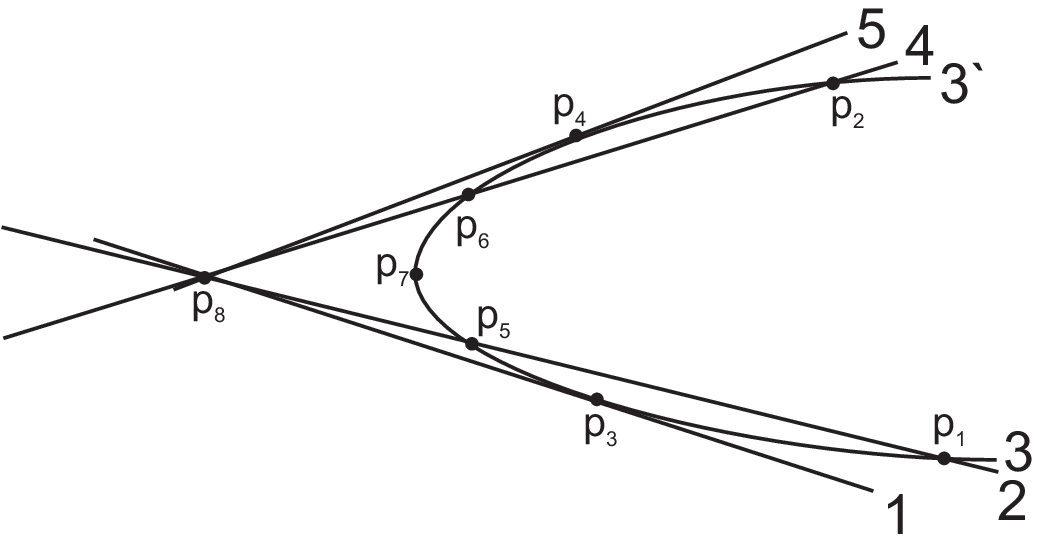}\\

\small{Figure 10} \end{center}
\begin{prs} \label{prs3_2} The local braid monodromy factorization of the above
configuration is
$$
\widetilde{\vp} = Z^2_{2,3} Z^2_{3',4} Z^4_{1,3} \bz^4_{3',5}
\wz^2_{3',4} \wz^2_{2,3} \wz^2_{3,3'}(\Dl^2
\langle1,2,4,5\rangle)^{Z^{-2}_{3,4}}
$$
where the braids $\wz_{3',4},\wz_{2,3}, \wz_{3,3'}$ correspond to
the following paths:

\begin{center}
\epsfig{file=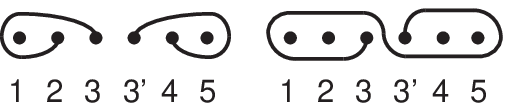}\\

\emph{\small{Figure 11} }\end{center}
\end{prs}

\begin{proof} Let $\{p_j\}^8_{j=1}$ be the singular
points of a small neighbourhood (that is $U$) of $v_{2,3}$ (see
figure 10) with
respect to $\pi_1$ (the projection to the $X$-axis) as follows:\\
$\{p_1, p_2\}, \{p_2, p_5\}$ -- the intersection points of $Q_3$
with $L_2, L_4$.\\
$p_3, p_4$ -- the tangent points of $Q_3$ and $L_1, L_5$.\\
$p_7$ -- the branch point of $Q_3$.\\
$p_8$ -- the intersection point of $\{L_i\}_{i = 1,2,4,5}$.

Let $E$ (resp. $D$) be a closed disk on the $X$-axis (resp.
$Y$-axis).  Let $N = \{x(p_j) = x_j | 1 \leq j \leq 8\},$ s.t. $N
\subset E - \partial E$.  Let $M$ be a real point on the $x$-axis,
s.t. $x_j \ll M, \forall x_j \in N, 1 \leq j \leq 8$.  There is a
$g$-base $\ell(\g_j)^8_{j=1}$ of $\pi_1(E - N,u)$, s.t. each path
$\g_j$ is below the real line and the values of $\vp_M$ with
respect to this base and $E \times D$ are the ones given in the
proposition.  We look for $\vp_M(\ell(\g_j))$ for $j = 1, \cdots ,
8$.  Choose a $g$-base $\ell(\g_j)^8_{j=1}$ as above and put all
the data in the following table:

\begin{center}
\begin{tabular}{l|c|c|c} $j$ & $\lm_j$ & $\varepsilon_j$ & $\dl_{
j}$ \\ \hline
1 & $<2,3>$ & 2 & $\Dl<2,3>$\\
2 & $<3',4>$ & 2 & $\Dl<3',4>$\\
3 & $<1,2>$ & 4 & $\Dl^2<1,2>$\\
4 & $<4,5>$ & 4 & $\Dl^2<4,5>$\\
5 & $<3',4>$ & 2 & $\Dl<3',4>$\\
6 & $<2,3>$ & 2 & $\Dl<2,3>$\\
7 & $<3,3'>$ & 1 & $\Dl^{1/2}_{IR}<2>$\\
8 & $<1,2,4,5>$ & 2 &  --\\
\end{tabular}
\end{center}

So, we get the following:\\
$\xi_{x_1} = z_{2,3}\\\varphi_M(\ell(\gamma_1)) = Z^2_{2,3}$\\[1ex]
$\xi_{x_2} = z_{3',4}$\,\,($\Delta\langle 2,3\rangle$ does not affect this
path)\\$\varphi_M(\ell(\gamma_2)) = Z^2_{3,4}$\\[1ex]
$\xi_{x_3} =\,$
\epsfig{file=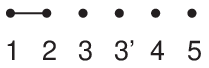}$\xrightarrow[\Delta<2,3>]{\Delta<3',4>}\,
\epsfig{file=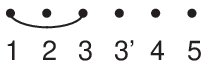} =z_{1,3}\\[1ex]
\varphi_M(\ell(\gamma_3)) = Z^4_{1,3}\\[1ex]
\xi_{x_4} = $ \epsfig{file=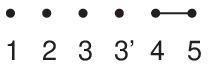}
$\xrightarrow[\Delta<3',4>]{\Delta^2<1,2>}\,$
\epsfig{file=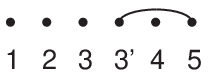} $= \bar{z}_{3,5}$ ($\Delta\langle
2,3\rangle$ does not affect this path)\\
$\varphi_M(\ell(\gamma_4)) = \bar{Z}^4_{3',5}\\[1ex]
\xi_{x_5} = $ \epsfig{file=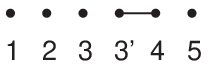}
$\xrightarrow[\Delta^2<1,2> \atop \Delta<3',4> ]{\Delta^2<4,5>}\,$
\epsfig{file=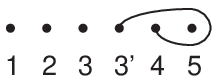} $=
\tilde{z}_{3',4}$\\[1ex]($\Delta\langle
2,3\rangle$ does not affect this path)\\
$\varphi_M(\ell(\gamma_5)) = \tilde{Z}^2_{3',4}\\[1ex]$
$\xi_{x_6} =$ \epsfig{file=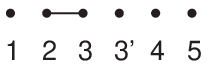}
$\xrightarrow[\Delta^2<4,5> \atop \Delta<1,2> ]{\Delta<3',4>}\,$
\epsfig{file=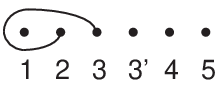}
$\xrightarrow[\Delta<2,3>]{\Delta<3',4>}\,$
\epsfig{file=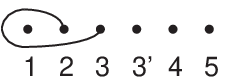} $ = \tilde{z}_{2,3}$\\[1ex]
$\varphi_M(\ell(\gamma_6)) = \tilde{Z}^2_{2,3}\\[1ex]$
$\xi_{x_7} =$ \epsfig{file=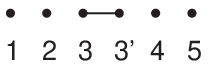}
$\xrightarrow[\Delta<3',4>]{\Delta<2,3>}\,$
\epsfig{file=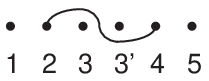}
$\xrightarrow[\Delta^2<1,2>]{\Delta^2<4,5>}\,$
\epsfig{file=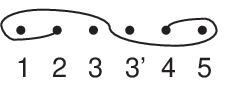}
$\xrightarrow[\Delta<2,3>]{\Delta<3',4>}\\[1ex]$
\epsfig{file=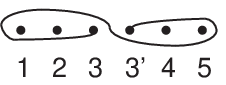} $= \tilde{z}_{3,3'}$\\[1ex]
$\varphi_M(\ell(\gamma_7)) = \tilde{Z}_{3,3'}\\[1ex]$
$\xi_{x_8} =$ \epsfig{file=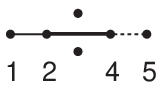}
$\xrightarrow{\Delta^{\frac{1}{2}}_{IR}<2>}$
\epsfig{file=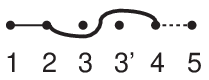}
$\xrightarrow[\Delta<3',4>]{\Delta<2,3>}\,$
\epsfig{file=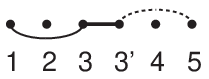}
$\xrightarrow[\Delta^2<1,2>]{\Delta^2<4,5>}\,$
\epsfig{file=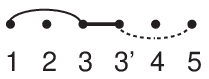}\\[1ex]
$\xrightarrow[\Delta<2,3>]{\Delta<3',4>}$
\epsfig{file=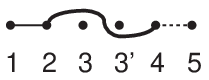} $= \Delta\langle
1,2,4,5\rangle_{Z_{\alpha}}$\\[1ex]
where $Z_{\a}$ is the braid induced from the motion
\epsfig{file=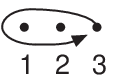} \\[1ex]
$\varphi_M(\ell(\gamma_8)) = \Delta^2\langle
1,2,4,5\rangle_{Z_{\alpha}}$

\end{proof}

The following regeneration regenerates a small neighborhood of
$\bigcup\limits_{i=1,2,4,5}L_i$, which is, by definition, a
4-point. Since this type of 4-point and its BMF of its
regeneration was treated earlier \cite{AT}, we can find out what
is the BMF of $v_{2,3}$ after the full regeneration.

\begin{prs} \label{prs3_3}
The local BMF $\vp_{2,3}$ around a small neighborhood of $v_{2,3}$
is:
\def\zmo{Z^{-1}}
\def\ztw{Z^2}
\def\zmtw{Z^{-2}}
\def\zt{Z^{(3)}}
\def\wzt{\widetilde{Z}^2}
\def\bzt{\bar{Z}^{(3)}}
$$
\vp_{2,3}  =  \ztw _{2',3} \ztw_{2,3} \ztw_{3',4'} \ztw_{3',4}
\zt_{11',3} \cdot \bzt_{3',55'} \wzt_{3',4'} \wzt_{3',4}
\wzt_{2',3} \wzt_{2,3}
$$
$$
\wz_{3,3'} (F_3 \cdot (F_3)_\vt)_{Z_{\alpha}}
$$
where $\vt = \zmo_{4,4'} \cdot \zmo_{5,5'}$, the braids
$\wz_{3',4'}, \wz_{3',4}, \wz_{2',3}, \wz_{2,3}, \wz_{3,3'}$
correspond to the following paths:

\begin{center}
\epsfig{file=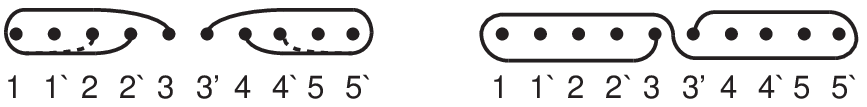}\\
\end{center}
and $Z_{\alpha}$ is the braid induced from the motion:
\begin{center}
\epsfig{file=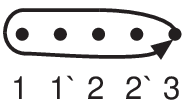}\\

\emph{\small{Figure 12} }\end{center} and
\def\ztw{Z^2}
\def\zmtw{Z^{-2}}
\def\zmo{Z^{-1}}
\def\zt{Z^{(3)}}
\def\wzt{\widetilde{Z}^2}
\def\bzt{\bar{Z}^{(3)}}
$$F_3  =  \zt_{22',4}  \ztw_{4',5}
(\ztw_{4,5})_{\ztw_{22',4}}
(\zt_{22',5})_{\ztw_{22',4}} (Z_{1,2'})_{\ztw_{2',5}\ztw_{2',4}}
(Z_{1',2})_{\ztw_{2,5} \ztw_{2,4}\ztw_{2,2'}} $$
\end{prs}

\begin{proof} Using the regeneration rules, we
replace
\begin{enumerate}
\item $Z^2_{2,3}\,\,\, (Z^2_{3`,4}, \tilde{Z}^2_{3`,4},
\tilde{Z}^2_{2,3})$ by $Z^2_{2\,2',3}\,\,\, (\mbox{resp.\,\,}
Z^2_{3`,4\,4'}, \tilde{Z}^2_{3`,4\,4'}, \tilde{Z}^2_{2\,2',3})$
(by the second regeneration rule) \item $Z^4_{1,3}\,\,\,
(\bar{Z}^4_{3',5})$ by $Z^{(3)}_{1\,1',3} \,\,\, (\mbox{resp.\,\,}
\bar{Z}^{(3)}_{3',5\,5'})$(by the third regeneration rule) \item
$\Dl^2\langle 1,2,4,5\rangle$ by $F_3 \cdot(F_3)_\vt$.
\end{enumerate}\end{proof}

\begin{remark} \label{rem3_1} Note that the last BMF was given when numerating the lines
in the neighbourhood of $v_{2,3}$ locally. So, when numerating
globally, we get:
$$
\vp_{2,3}  =  \ztw_{3',9} \ztw_{3,9} \ztw_{9',1\,6'} \ztw_{9',16}
\zt_{1\,1',9} \bzt_{9',21\,21'} \wzt_{9',16'} \wzt_{9',16}
\wzt_{3',9} \wzt_{3,9}
$$
$$\wz_{9,9'} \cdot (F_3 \cdot (F_3)_\vt)_{Z_{\alpha_3}} $$ where $\vt =
\zmo_{16, 16'} \zmo_{21,21'}$,
$Z_{\alpha_3}$ is the braid induced from the motion:
\begin{center}
\epsfig{file=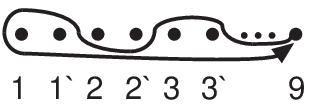}\\

\emph{\small{Figure 13}}
\end{center}
and
$$
F_3  =  \zt_{3\,3',16} \ztw_{16',21}
(\ztw_{16,21})_{\ztw_{3\,3',16}}
(\zt_{3\,3',21})_{\ztw_{3\,3',16}}
(Z_{1,3'})_{\ztw_{3',21}\ztw_{3',16}} \cdot
(Z_{1',3})_{\ztw_{3,21} \ztw_{3,16}\ztw_{3,3'}}. $$\end{remark} We
now write the other BMFs.

\begin{prs} \label{prs3_4}
The local braid monodromy $\vp_{2,1}$ is:
\def\zmo{Z^{-1}}
\def\ztw{Z^2}
\def\zmtw{Z^{-2}}
\def\zt{Z^{(3)}}
\def\wzt{\widetilde{Z}^2}
\def\bzt{\bar{Z}^{(3)}}
\def\bztw{\bar{Z}^2}
$$
\vp_{2,1}  =  \zt_{11 \, 11',15}
(F_u(11,5,1,7))^{\ztw_{11\,11',15}}\overset{(7-7' \atop
5-5')}{Z_{11',15}} \wz_{15, 15'}
\overset{(7-7')}{\ztw_{5\,5',15\,15'}} \wzt_{7\,7',15'}
\ztw_{7\,7',15}
$$
where $ \wz_{15, 15'} , \wz_{7\,7,15'} $ correspond to the
following paths:

\begin{center}
\epsfig{file=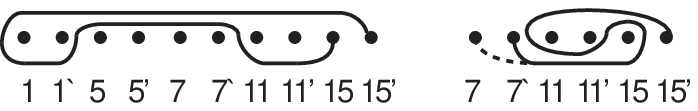}\\

\emph{\small{Figure 14} }\end{center}

The local braid monodromy $\vp_{2,2}$ is:
$$
\vp_{2,2}  =  \ztw_{12 \, 12',20} \zt_{8 \, 8',20} \wzt_{12
\,12',20}  (F_u(8,6,2,12))_{Z_{\alpha_2}}
\overset{(6-6')}{\zt_{2\,2, 20'}} \wzt_{20,20'}
 \wzt_{6\,6',20'} \ztw_{6\,6', 20}
$$
where $\wz_{12 \,12',20},\wz_{20,20'} ,\wz_{6\,6',20'}$ correspond
to the following paths:\\

\begin{center}
\epsfig{file=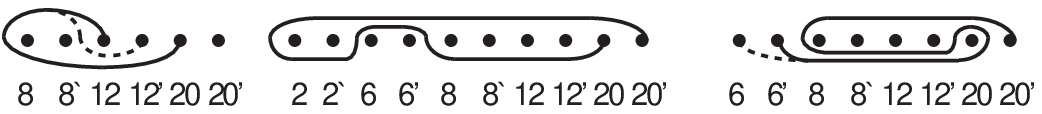}\\
\emph{\small{Figure 15}}

\end{center}
and $Z_{\alpha_2}$ is the braid induced from the motion:
\epsfig{file=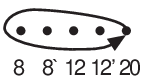}.\\

The local braid monodromy $\vp_{2,4}$ is:
$$\vp_{2,4}  =
\ztw_{4\,4',13} \zt_{13',17\,17'} \zt_{2\,2',13} \ztw_{4\,4',13}
\wz_{13,13'} \overset{(17-17')}{\underline{Z}^2_{13',22\,22'}}
\wzt_{13,22\,22'} (F_u(22,4,2,17))_{Z_{\alpha_4}}
$$
where $\wz_{13,13'}, \wz_{13',22\,22'}$ correspond to the
following paths:

\begin{center}
\epsfig{file=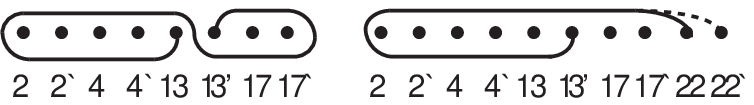}\\
\emph{\small{Figure 16} }\end{center} and $Z_{\alpha_4}$ is the
braid induced from the motion:
\epsfig{file=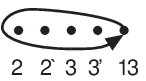}.\\

\def\zmo{Z^{-1}}
\def\ztw{Z^2}
\def\zmtw{Z^{-2}}
\def\zt{Z^{(3)}}
\def\wzt{\widetilde{Z}^2}
\def\bzt{\bar{Z}^{(3)}}
\def\bztw{\bar{Z}^2}

The local braid monodromy $\vp_{2,7}$ is:
$$
\vp_{2,7}  =  \ztw_{8\,8',9} \ztw_{10 \,10',9'} \zt_{7\,7',9}
\overset{\hspace{-1cm}{\scriptscriptstyle{(10-10')}}}{\uz^{(3)}_{9',
24\, 24'}}   \wzt_{9',10\,10'} \wzt_{8\,8',9} \wz_{9,9'}
(F_\ell(7,10,24,8))_{Z_{\alpha_7}}
$$
where $\wz_{9',10\,10'}, \wz_{8\,8',9'}, \wz_{9,9'}$ correspond to
the following paths:
\begin{center}
\epsfig{file=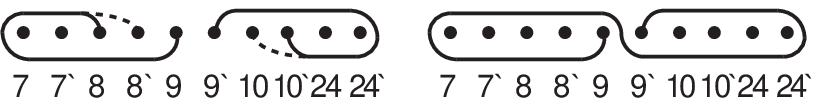}\\
\emph{\small{Figure 17} }\end{center} and $Z_{\alpha_7}$ is the
braid induced from the motion:
\epsfig{file=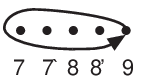}.\\

The local braid monodromy $\vp_{2,8}$ is:\\
$$
\vp_{2,8}  =  \ztw_{13',14\,14'} \zt_{12\,12',13}
\overset{\hspace{-1cm}{\scriptscriptstyle{(14-14')}}}{\uz^{(3)}_{13',
19\, 19'}} \wzt_{13',14\,14'} \wz_{13,13'} \wzt_{11\,11',13'}
\ztw_{11\,11',13} F_\ell(12,14,19,11))_{Z_{\alpha_8}}
$$
where $\wz_{13',14\,14'}, \wz_{11\,11',13'}, \wz_{13,13'}$
correspond to the following paths:

\begin{center}
\epsfig{file=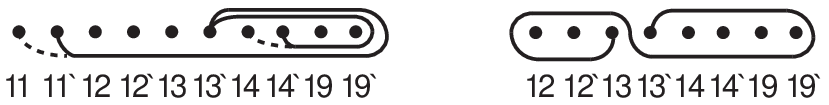}\\
\emph{\small{Figure 18} }\end{center} and $Z_{\alpha_8}$ is the
braid induced from the motion:
\epsfig{file=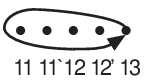}.\\

The local braid monodromy $\vp_{2,9}$ is:
$$\vp_{2,9}  =
\zt_{15',16\,16'} F_\ell(17,19,18,16))_{\ztw_{15',16\,16'}}
\overset{\hspace{-1cm}{\scriptscriptstyle{(16-16')}}}{\uz^{(3)}_{15',
19\, 19'}} \wz_{15\,15'} $$$$\wzt_{15,18\,18'}
\overset{\hspace{-1cm}{\scriptscriptstyle{(16-16')}}}{Z^2_{15',
18\,18'}}  \bztw_{15',17\,17'}  \wzt_{15, 17\,17'}
$$
where $\wz_{15\,15'}, \wz_{15\,18\,18'}, \wz_{15,17\,17'}$
correspond to the following paths:

\begin{center}
\epsfig{file=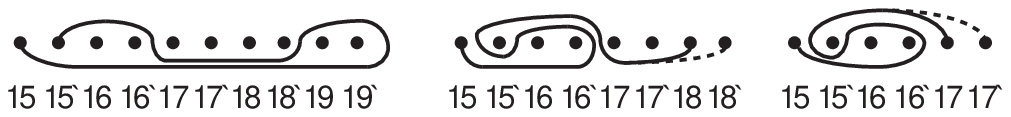}\\
\emph{\small{Figure 19}}
\end{center}

The local braid monodromy $\vp_{2,10}$ is:
$$ \vp_{2,10}  =
\ztw_{20',21\,21'} \bzt_{20',22\, 22'} \wzt_{20',21\, 21'}
(F_\ell(21,24,23,22))_{Z_{\alpha_{\!10}}}
\underset{{\hspace{-1cm}{\scriptscriptstyle{(23-23')}}}}{\bz^{(3)}_{20',24\,24'}}$$$$
\wz_{20\,20'} \bztw_{20',23\,23'}  \wzt_{20, 23\,23'}
$$
where $\wz_{20', 21\,21'}, \wz_{20\,20'} , \wz_{20, 23\,23'}$
correspond to the following paths:

\begin{center}
\epsfig{file=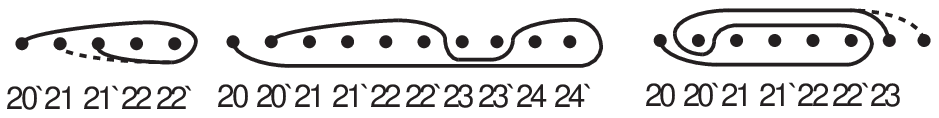}\\

\emph{\small{Figure 20} }
\end{center}
and $Z_{\alpha_{10}}$ is the braid induced from the motion:
\epsfig{file=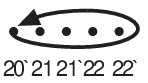}.\\
\end{prs}

Performing the regeneration affects also the braids induced from
the parasitic line intersection.  Denote by $C_{2,i}$ the braid,
which is created from $\wc_{2,i}$ during the regeneration process.

Every $\wc_{2,i}$ is a product of a 2-degree braid $\ztw_{i,j}$,
which becomes, as a consequence of the second regeneration rule,
an 8-degree braid:   $\ztw_{ii',jj'} = \ztw_{i',j'} \ztw_{i',j}
\ztw_{i,j'} \ztw_{i,j}$.  If the path representing the braid
$\ztw_{i,j}$ was above/below a point $p$, then the induced braids
would be above/below the points $p$ and $p'$.

Before we present the global BMF, we have to check if there are
extra branch points in the branch curves, that are created during
the regeneration of a line $L_i$. An extra branch point
contributes to a factorization the factor $Z_{i,i'}$. (By
``contributes'' we mean that one should multiply the old
factorization $Z_{i,i'}$ from the right).
\\
\underline{$X_1$}: It was proven in \cite[prop. 16]{ATCM} that the
factorization $\vp_1 = \pl^{1}_{j=10} C_{1,j} \vp_{1,j}$ is a BMF
of the branch curve of $X_1$.  Thus, there are no missing braids
in the factorization above, and therefore there are no extra
branch points.
\\
\underline{$X_2$}: Denote by $\widetilde{\Delta} = \pl^{1}_{j=10}
C_{2,j} \vp_{2,j}$.  If $\widetilde{\Delta}$ was the BMF of the
branch curve of $X_2$, then deg($\widetilde{\Delta}) =
\mbox{deg}(\Delta^2_{48}) = 48 \cdot (48 -1) = 2256$.  We show
that this is not the situation here. $ \mbox{deg}
(\widetilde{\Delta}) = \sum\limits^{10}_{j=1} \mbox{deg}(C_{2,j})
+ \sum\limits^{10}_{j=1}\mbox{deg} (\vp_{2,j})$.  $\sum
\mbox{deg}(C_{2,j}) = 8 \cdot 184 = 1472 \ . \ \mbox{For} \ j=5,6,
\ v_{2,j}$ are 4-point, and by \cite{MoTe4}, deg$(\vp_{2,5}) =
\mbox{deg}(\vp_{2,6}) = 48 \,.\, \mbox{For} \ 1 \leq j \leq 10, \
j \neq 5,6 \ v_{2,j}$ are 5-point.  Although these points have
different configurations, their BMFs $- \vp_{2,j}$ still have 6
factors of degree 3,\ 8 factors of degree 2, one factor of degree
1, and a factor representing the BMF of the regeneration of a
4-point, whose degree is 48.  Thus $\forall \ 1 \leq j \leq 10, \
j \neq 5,6, \ \mbox{deg}(\vp_{2,j}) = 6 \cdot 3 + 8 \cdot 2 + 1 +
48 = 83$.  So, deg$(\widetilde{\Delta}) = 1472 + 2 \cdot 48 + 8
\cdot 83 = 2232 < 2256$.

Define the forgetting homomorphisms:
$$
1 \leq i \leq 24\, \, f_i : B_{48} [D, \{1,1', \cdots , 24,24' \}]
\rightarrow B_2 [D, \{i,i' \}].
$$
It is clear that if $\widetilde{\Delta}$ was a BMF, then $\forall
\ i, \mbox{deg}(f_i(\widetilde{\Delta})) = 2$.  However, this is
not the case in the current situation.  It was proven in
\cite{RobbT} (see also \cite{Robb}), that if
deg$(f_i(\widetilde{\Delta})) = k < 2$, then there are $(2-k)$
extra branch points, and so there is a contribution of the
factorization $\prod\limits^{2-k}_{m=1} Z_{i,i'}$ to
$\widetilde{\Delta}$.

\begin{prs} \label{prs3_5} \begin{enumerate}
\item[\emph{(1)}]  The regeneration of the lines $L_{2,j}, \ j =
3,4,5,6,10,14,18,23$ contributes the factors $Z_{j,j'} \cdot
Z_{j,j'}$ to $\widetilde{\Delta}$. \item[\emph{(2)}]  The
regeneration of the lines $L_{2,j} \ j = 7,8,11,12,16,17,21,22$
contributes the factor $Z_{j,j'} $\linebreak to
$\widetilde{\Delta}$.\end{enumerate}
\end{prs}

\begin{proof} (1) We prove this case for $j=3;$ the other cases are
done using the same method. By Lemma 3.3.3 (or Proposition 3.3.4)
in \cite{Robb}, it is enough to prove that
deg$(f_3(\widetilde{\Delta})) = 0$.  The braids coming from the
parasitic intersection are sent by $f_3$ (and by any $f_i$, in
fact) to $Id$, so it is enough to look only at the factors
$\vp_{2,k}, \ 1 \leq k \leq 10$ that involve braids, one of whose
end points are  3 or $3'$.  The only suitable $k's$ are $k=5$ and
$k=3$. Since $v_{2,3}$ and $v_{2,5}$ are both of 4-point, by
\cite[Lemma 8, (iv)]{MoTe4}, deg$(f_3(\vp_{2,3}) =
\mbox{deg}(f_3(\vp_{2,5})) = 0$.  Therefore
deg$(f_3(\widetilde{\Delta})) = 0$.\\
(2)  We prove for $j=7$; the other cases are done using the same
method.  It is enough to prove that deg$(f_7(\widetilde{\Delta}))
= 1$  (by \cite{Robb}).

As in (1), we only consider the factors $\vp_{2,1}$ and
$\vp_{2,7}$. $v_{2,1}$ is a 5-point.  The first regeneration is of
the line $L_{2,15}$, (which turns into a conic, that intersects
the line $L_{2,7}$ at two nodes, which induce braids of the form
$Z^2_{7,15}$ and $Z^2_{7,15'}$), which does not contribute to the
regeneration factors of the  form $Z_{7,7'}$.  After this
regeneration, we are left with the regeneration of a 4-point, and
by \cite[Lemma 8, (iv)]{MoTe4}, we get deg$(f_7(\vp_{2,1})) = 0$.

$v_{2,7}$ is also a 5-point.  The first regeneration is of the
line $L_{2,9}$, which turns into a conic, $ Q_{2,(9,9')}$, that is
tangent to $L_{2,7}$ (by \cite[Claim 1]{MoTe4}).  This tangency
point is regenerated into three cusps (see \cite{MoTe2}) which
induces the product of three braids - $Z^3_{7',9} \cdot Z^3_{7,9}
\cdot (Z^3_{7',9})_{Z^{-1}_{7,7'}} =: Z^{(3)}_{7\,7',9}$. By
\cite[Lemma 2, (i)]{MoTe4}, we see that
deg$(f_7(Z^{(3)}_{7\,7',9})) = 1$. Again, the regeneration
afterwards of the 4-point does not contribute a factor of the form
$Z_{7,7'}$ to the factorization. Thus, we get deg$(f_7(\vp_{2,7}))
= 1$, and deg$(f_7(\widetilde{\Delta})) = 1$. \end{proof}

Define an ordered set
$$\{i_n\}^{16}_{n=1} :=
\{3,4,5,6,7,8,10,11,12,14,16,17,18,21,22,23 \},
$$
$\mbox{and for} \ 1 \leq n \leq 16$ let:
$$
b_n = \left\{ \begin{array}{llll} Z_{i_n,i'_n} \cdot Z_{i_n,i'_n}
& i_n = \{3,4,5,6,10,14,18,23 \} \\
Z_{i_n,i'_n}
& i_n = \{7,8,11,12,16,17,21,22 \}
\end{array} \right .
$$

\begin{prs} \label{prs3_6}
$$
\vp_2 = \pl^{1}_{j=10} C_{2,j} \vp_{2,j} \pl^{16}_{n=1} b_n
$$
is a braid monodromy factorization for $S_2$.
\end{prs}

 The proof is divided into a number of
lemmas.

\begin{lemma} \label{lem3_1} $ \vp_2 = \pl^{1}_{j=10} C_{2,j} \widetilde{\vp}_{2,j}
\pl^{16}_{n=1} b_n $ is a braid monodromy factorization for $S_2$,
where $\widetilde{\vp}_{2,j} = (\vp_{2,j})_{h_j}$ for some $h_j
\in \langle Z_{kk'} | v_{2,j} \in L_{2,k}  \rangle.$
\end{lemma}
\begin{proof} Using Proposition VI.2.1 from
\cite{MoTe1} on $S_2$, we get that  $ \vp_2 = \pl^{1}_{j=10}
C_{2,j} \widetilde{\vp}_{2,j} \pl b_\ell $.\linebreak $h_j \in
\langle Z_{kk'} | \, v_{2,j} \in L_{2,k} \rangle$ are determined
by the regeneration of the embedding $B_k \hookrightarrow B_{24}$
to $B_{2k}\hookrightarrow B_{48}$ where $k=4$ when $j = 5,6$ and
$k=5$ otherwise $(1 \leq j \leq 10, \ j \neq 5,6 ;$  see the
definition of regeneration of an embedding in \cite[Sec.
1]{MoTe4}). $b_\ell$ are factors that are not converted by $\pl
C_{2,j} \widetilde{\vp}_{2,j}$, and each $b_n$ is of the form
$Y^{t_i}_i, Y_i$, is a positive half-twist, $0 \leq t_i \leq 3$.
Note that deg$(\widetilde{\vp}_{2,j}) = \mbox{deg}(\vp_{2,j})$. By
the previous proposition, we know part of the $b_\ell$'s; so we
can say that $\vp_{2} = \pl^{1}_{j=10} C_{2,j}
\widetilde{\vp}_{2,j} \pl^{16}_{n=1} b_n \prod b_\ell.$ We compute
deg$\bigg(\pl^{1}_{j=10}C_{2,j} \widetilde{\vp}_{2,j}
\pl^{16}_{n=1} b_n\bigg)$.  By earlier computations and the
previous proposition,
$$
\mbox{deg}\bigg(\pl^{1}_{j=10}C_{2,j} \widetilde{\vp}_{2,j}
\pl^{16}_{n=1} b_n\bigg) = 2232 + 2 \cdot 8 + 8 = 2256 = 48 \cdot
47 \ = \mbox{deg} (\Dl^2_{48}) = \mbox{deg} \vp_2.
$$
Thus, we have to compute deg$(\prod b_\ell)$.  Since $\forall
\ell, b_\ell$ is a positive power of a positive half-twist, we get
$b_\ell = 1 \, \forall \ell$.  So we have
$$
\vp_2 = \pl^{1}_{j=10} C_{2,j} \widetilde{\vp}_{2,j}
\pl^{16}_{n=1} b_n.
$$\end{proof}
\begin{lemma}  \label{lem3_2}$\vp_2 = \pl^{1}_{j=10} C_{2,j}
\vp_{2,j} \pl^{16}_{n=1} b_n.$ \end{lemma}

\begin{proof} Recalling the invariance rules for the
BMF of 4- and 5- point (see \cite{MoTe4} and the appendix (Section
(\ref{sec5}))), we can apply them as in \cite{MoTe4}, and get that
$\pl^{1}_{j=10} C_{2,j} \vp_{2,j} \pl^{16}_{n=1} b_n$ is also a
braid monodromy factorization.

Note that although the invariance rules for the 5-point are
different from the invariance rules of the standard 4/6 - point,
what matters, as can be seen in \cite[Section 4]{MoTe4} is that
the invariance rule regarding the horizontal lines in the 5-point
(the two lines that are regenerated last) remains the same in this
type of point. \end{proof}

\section[Computing the fundamental groups]{Computing the fundamental groups}

\subsection[Computation for $X_2$]{Computation for $X_2$} \label{sec4_1}

By the Van Kampen theorem (Theorem (\ref{thm2_5})), we can compute
the relations between the generators in the fundamental group of
the complement of the branch curve.

We will prove that $\pi_1(\mathbb{C} - S_2)$ is a quotient of
$\tilde{B}_{16}$. In order to do so, we have to compute the local
relations (or the local fundamental groups of the complement of
the branch curve) arising from each singular point of the branch
curve.  Note that points $v_{2,5}, v_{2,6}$ are of the type
4-point, which was investigated in \cite{Mo}, \cite{RobbT}. Thus,
we have to look at the remaining 5-points. We focus only on one
5-point -- $v_{2,3}$; for the other 5-points, the procedure for
deducing the relations is the same, and we state (later) only the
relations coming from the branch points for these points.

Recall that in the regeneration process, every line is
``doubled'', and thus $S_i \cap \mathbb{C}$  will contain $48 = 2
\cdot 24$ points. The generators of $\pi_1 (\mathbb{C}^2 - S,u)$
(see the Van Kampen Theorem (\ref{thm2_5})) induced from this
doubling are denoted as $\{\G_i, \G_{i'} \}^{24}_{i=1}$, where
each pair
$\{\G_i, \G_{i'} \}$ originates from the same line.\\

Denote $\bG_i = \G_i$ or $\G_{i'}$. Before examining $\vp_{2,3}$
we state the following \\ \noindent
\begin{remark}\,\emph{$\forall_{i,j}$ s.t. $L_{2,i}
\bigcap L_{2,j} = \emptyset$, we have the following relations in\\
$\pi_1 (\mathbb{C}^2 - S_2)$:
$$
[\bG_i, \bG_j ] = 1.
$$}\end{remark}

The proof of this remark is based on the parasitic intersection
braids.  From each braid of the expressions $C_{2,i} \,(i = 1,
..., 10)$, using complex conjugation and the Van Kampen Theorem,
we can induce the above relations.\\

\begin{prs} \label{prs4_1}
The following relations in $\pi_1 (\mathbb{C}^2 - S_2)$ are
induced from $\vp_{2,3}$:
\\
\\
\emph{(1)} \  $<\bG_{16}, \bG_1> = <\bG_1, \bG_9> = < \bG_9,
\bG_{21}> = <\bG_{21}, \bG_3> = <\bG_3, \bG_{16} > = 1$
\\
\\
\emph{ (2)} \ $[\bG_i, \bG_j] = 1$ where $L_i, L_j$ do not bound
the same triangle, besides \\ $(i = 1', j = 3),\ ( i = 1, j = 3')$
and $(i = 1, j = 21)$
\\
\\
\emph{ (3)} \ $\G_1 = (\G_{3'})_{\G^{-1}_{16} \G^{-1}_{21}\G_9}.$

\end{prs}

\begin{proof}  In the course of the proof we use the
Van Kampen Theorem, the invariance relations of the 5-point, and
the complex conjugation method (see \cite{MoTe4}).  We prove the
proposition in several steps.
\\
\\
\noindent \textsl{Step 1}:  By looking at the braids (in
$\vp_{2,3})\ \ztw_{3\,3', 9}, \ztw_{9\,9',16} , \zt_{1\,1',9}$ and
$\bzt_{9',21\,21'}$, we induce immediately (using invariance
relations and complex conjugation for the last braid) the
following relations:
$$[\bG_9, \bG_{16}] = [ \bG_3, \bG_9] = <\!
\bG_1, \bG_9 \! > = <\! \bG_9, \bG_{21}\!> = 1.
$$
\noindent \textsl{Step 2}:  Note that the factors in $(F_3 \cdot
(F_3)_\vartheta)$ are conjugated by $Z_{\alpha_3}$. Denote the
corresponding generators induced from $(F_3 \cdot
(F_3)_\vartheta)$ (after the conjugation) by $\wG_i$.

So:
$$
\wG_{3} =\G_9 \G_{3} \Gmo_9 \stackrel{\mathrm{by \ step \ 1}}{=}
\G_{3} \quad \quad \wG_{3'} = \G_9 \G_{3'} \Gmo_9 = \G_{3'}
$$
$$
\wG_{1} =\G_9 \G_{1} \Gmo_9  \quad  \quad \quad \, \quad \wG_{1'}
= \G_9 \G_{1'} \Gmo_9.
$$
the other $\wG_i$-s are not changed.  So, we have, by the braid
$\zt_{3\,3', 16}$ in $F_3$  the relation
$$
< \bG_3, \bG_{16} > = 1
$$
\noindent \textsl{Step 3}:  From the braid $\ztw_{16',21}$ in
$F_3$, we get the relation: $[\G_{16'}, \G_{21}] = 1$.  Looking on
the complex conjugate of the braid $(\ztw_{16',21})_\vartheta$, we
now get the relation
$$
[\G_{16}, \Gmo_{21} \G_{21'} \G_{21}] = 1 \stackrel{\mathrm{(inv.
rel.}\,\rho_{16}\rho_{21})}{\longrightarrow} [\G_{16'}, \Gmt_{21}
\G_{21'} \G^2_{21}] = [\G_{16'} , \G_{21'}] = 1.
$$
By performing another time the invariance relation $(\rho_{16}
\rho_{21})$, we get $[\G_{16}, \G_{21}] = 1$.   From \linebreak
$[\G_{16'}, \Gmo_{21} \G_{21'} \G_{21}] = 1$, we get $[\G_{16},
\G_{21'}] = 1$.

So we have the relation $[\bG_{16}, \bG_{21}] = 1$.
\\

\noindent \textsl{Step 4}: \ From the braid
$(\zt_{3\,3',2\,1})_{\zt_{3\,3',16}}$, we get the relation: \\
$<\G_{21}, \G_{16} \G_{3'} \Gmo_{16} > = 1$. By Step 3 we get $<
\G_{21}, \G_{3'} > = 1$; in the same way, we get $< \G_{21},
\G_{3'} > = 1$ and by invariance relation, we get: $<\bG_{21},
\bG_3 > = 1$.
\\

\noindent \textsl{Step 5}: \ From the braid
$((Z_{1,3'})_{\ztw_{3',21}\ztw_{3',16}})_{Z_{\alpha_3}}$, we get
the
relation: \\
$\G_{1} = (\G_{3'})_{\Gmo_{16} \Gmo_{21} \G_9}. $ Thus
\begin{eqnarray*}
<\G_{16}, \G_1> & = & <\G_{16}, (\G_{3'})_{\Gmo_{16}
\Gmo_{21}\G_9}
>=
<(\G_{16})_{\G_{21}} , (\G_{3'})_{\Gmo_{16}} > \\
& \stackrel{[\G_{16}, \G_{21}]=1}{=} & <\G_{16}, \G_{16} \G_{3'}
\Gmo_{16} > = <\G_{16}, \G_{3'} > = 1.
\end{eqnarray*}
By the invariance relations, we get: $<\bG_{16}, \bG_1 > = 1$.
\\

\noindent \textsl{Step 6}:  We know that $\G_1 =
(\G_{3'})_{\Gmo_{16} \Gmo_{21} \G_9}$ and  thus $(\G_1)_{\Gmo_{9}
\G_{16}}  = (\G_{3'})_{\Gmo_{21}}$ (by $[\G_{16}, \G_{21}]=1)$.
From the braid $((\ztw_{16,21})_{\ztw_{3\,3',16}})_{Z_{\alpha_3}}$
, we get the relation:
\begin{eqnarray*}
&& [\G_{16}, (\G_{21})_{\G_{3'} \G_3}] = 1 \ \mbox{or} \\
1 & = & [\G_3 \G_{16} \Gmo_3, \Gmo_{3'} \G_{21} \G_{3'}]
\stackrel{<\G_{3}, \G_{16}>=<\G_{3'}, \G_{21}> = 1}{=} \\
&& [\Gmo_{16} \G_3 \G_{16} , \G_{21} \G_{3'} \Gmo_{21} ] = \\
&& [\Gmo_{16} \G_3 \G_{16} , \Gmo_{16} \G_9 \G_{1} \Gmo_{9}
\Gmo_{16} ] \stackrel{[\G_9, \G_{16}]=[\G_9,\G_3]=1}{=} [\G_3,
\G_1]
\end{eqnarray*}
and by invariance we get $[\G_{3'} , \G_{1'} ] = 1$. \end{proof}

The following proposition proves the missing relations (e.g.,
$[\bG_1, \bG_{21}] = 1)$.  The reason for separating this
proposition from the former is because we use now relations which
are not necessarily from $\vp_{2,3}$.

\begin{prs} \label{prs4_2}
The following relations in $\pi_1 (\mathbb{C}^2 - S_2)$ hold
$$
[\bG_1, \bG_3] = [\bG_1, \bG_{21}] = 1.
$$
\end{prs}

\begin{proof}  Due to the invariance relations of
$v_{2,3}$, it is enough to prove $[\G_1, \G_{3'}] = 1$ and $[\G_1,
\G_{21}] = 1$.

By the braid $Z_{3,3'}$ (induced from an extra branch point), we
know that $\G_3 = \G_{3'}$.  Thus, by the last proposition
((\ref{prs4_1}), Step 6), we have
$$
1 = [\G_1, \G_3] = [\G_1, \G_{3'}] .
$$
Looking on the local BMF of $v_{2,1}$, we have the following
relation from the braid $((Z_{1,5'})_{\ztw_{5',11}
\ztw_{5',7}})_{\ztw_{11 \, 11',15}}:$
$$
\G_1 = (\G_{5'})_{\Gmo_7 \G_{15} \Gmo_{11} \Gmo_{15}}.
$$

Since $\G_{21}$ commutes with $\G_{5'}, \G_7, \G_{15}$ and
$\G_{11}$ (due to the parasitic intersection braids), we have that
$[\G_1, \G_{21}] = 1$. \end{proof}

\begin{prs} \label{prs4_3}
$\forall i, 1 \leq i \leq 24, i \neq 9, 13,15,20, \ \G_i =
\G_{i'}$ in  $\pi_1 (\mathbb{C}^2 - S_2)$.
\end{prs}
 We divide the proof into 2 lemmas.

\begin{lemma} For $i = 3,... , 8,10,11,12,14,16,17,18,21,22,23: \
\G_i = \G_{i'}$. \end{lemma}

\begin{proof}  The relation $\G_i = \G_{i'}$ is
induced from the braids $Z_{i,i'}$ which are created from the
extra branch points (by Proposition 3.6).\end{proof}

\begin{lemma}  For $i = 1,2,19,24: \ \G_i =
\G_{i'}$. \end{lemma} \begin{proof} We will prove in details only
for $i = 1$; the proof for the other $i$'s is the same.  We know
(from the braid
$((Z_{1,5'})_{Z^2_{5',7}Z^2_{5',11}})_{Z^2_{11\,11',15}}$ in
$\varphi_{2,1}$) the relation: $\G_1 = \G_{15} \G_{11} \G_7 \G_5
\Gmo_7 \Gmo_{11} \Gmo_{15}$ (we used the relation $\G_5 =
\G_{5`}$). Operating  the invariance relations $(\rho_{1}
\rho_{5})(\rho_7 \rho_{11})$ and taking the inverse, we get:
$$
\Gmo_{1`}= \G_{15} \G_{11'} \G_{7'} \Gmo_{5'} \Gmo_{7'} \Gmo_{11'}
\Gmo_{15}.
$$ Multiplying the above relations and using Lemma 4.1, we get
$\Gmo_{1'} \G_{1} = 1$, or $\G_1 = \G_{1'}$.

For $i = 2$, we use the braid $((Z_{2,6'})_{\ztw_{6',8}
\ztw_{6',12}})_{Z_{\alpha_2}}$ from $\vp_{2,2}$ and the same
method as above.

For $i = 19, 24$, one can use the braids
$((\bar{Z}_{10',24})_{\cdots})_{Z_{\alpha_7}}$ from
$\varphi_{2,7}$ (or the braid
$((\bar{Z}_{14',19})_{\cdots})_{Z_{\alpha_8}}$ from
$\varphi_{2,8}$) and continue as above. \end{proof}

\begin{remark} \label{rem4_2} \emph{For each $1 \leq i \leq 10$ we denote by $G_{2,i}$ the
local fundamental whose generators are $\G_j$, such that one of
the endpoints of $L_{2,j}$ is $v_{2,i}$.  Generalizing
Propositions 4.1 and 4.2, it is easy to prove that $\forall_{i,j}$
s.t. $L_{2,i}$ and $L_{2,j}$ do not bound a common triangle,
$[\bG_i, \bG_j] = 1$; and $\forall_{i,j}$ s.t. $L_{2,i}$ and
$L_{2,j}$ bound a common triangle, $< \bG_i, \bG_j > = 1$ (in
$\pi_1(\mathbb{C}^2 - S_2))$}. \end{remark}

\begin{remark} \label{rem4_3}\emph{It is important to state which braids are coming from
the branch points.  We list below (for each $\vp_{2,i}$, for $1
\leq i \leq 10, \ i \neq 5,6)$ which braid is induced from a
branch point, that is created during the regeneration of the
horizontal lines of the 5-point.  We use the double and triple
relations, and the last proposition, and we obtain:}.
\end{remark}

\begin{eqnarray*}
& i = 1: & \quad \G_1 = (\G_5)_{\Gmo_7 \Gmo_{11} \Gmo_{15}}\\
& i = 2: & \quad \G_2 = (\G_6)_{\Gmo_8 \Gmo_{12} \Gmo_{20}}\\
& i = 3: & \quad \G_1 = (\G_3)_{\Gmo_{16} \Gmo_{21} \G_9 }\\
& i = 4: & \quad \G_2 = (\G_4)_{\Gmo_{17} \G_{13} \Gmo_{22}}\\
& i = 7: & \quad \G_{24} = (\G_{10})_{\Gmo_{7} \Gmo_{8} \Gmo_9}\\
& i = 8: & \quad \G_{19} = (\G_{14})_{\G_{11} \G_{12} \G_{13} }\\
& i = 9: & \quad \G_{18} = (\G_{17})_{\Gmo_{19} \G_{15'} \Gmo_{16} }\\
& i = 10: & \quad \G_{23} = (\G_{24})_{\G_{20'} \Gmo_{22}
\G_{21}}.
\end{eqnarray*}

\begin{prs} \label{prs4_4}  For $i = 5,6$, there exist a homomorphism $\a_i :
\widetilde{B}_4 \rightarrow G_{2,i}.$
\end{prs}
\begin{proof} This proposition is proven in
\cite{RobbT}.\end{proof}

\begin{prs} \label{prs4_5} For $1\leq i \leq 10,\, i \neq
5,6$, there exist a homomorphism $\a_i : \widetilde{B}_5
\rightarrow G_{2,i}.$ \end{prs}

\begin{proof} Using the Remark (\ref{rem4_3}) we prove
only for $i = 1$, and the proof for the other $i$'s is done in the
same way.

It is easy to check that $\a_1 : B_5 \rightarrow G_{2,1}$ is
well-defined:
$$
\a_1(X_1) = \G_7 \quad \quad \a_1(X_2) = \G_5 \quad \quad
\a_1(X_3) = \G_{11} \quad \quad \a_1(X_4) = \G_{15}.
$$
Let $x_1, ..., x_4$ be the images of $X_1, ... , X_4$ in
$\widetilde{B}_5$.  Consider\\ $T = X_4 X_3 X_1 X_2 X^{-1}_1
X^{-1}_3 X^{-1}_4$ in $B_5$ (see the following figure):

\begin{center}
\epsfig{file=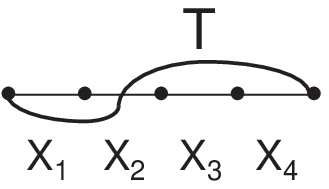}\\

\small{Figure 21}
\end{center}

$T$ is transversal to $X_2$.  Let $t$ be the image of $T$ in
$\widetilde{B}_5$; by the definition of $\widetilde{B}_5$ we have
$[t,x_2]=1$.  To show that $\a_1: B_5 \rightarrow G_{2,1}$ defines
the desired $\a_1 : \widetilde{B}_5 \rightarrow G_{2,1}$, it is
enough to check that
$$
[\a_1(T), \a(X_2)]=1.
$$

We claim that $\a(T) = \G_1$, because
$$
\a(T) = \a(X_4 X_3 X_1 X_2 X^{-1}_1 X^{-1}_3 X^{-1}_4) =
(\G_5)_{\Gmo_7 \Gmo_{11} \Gmo_{15}} = \G_1.
$$
So we have $[\a(T), \a(X_2)] = [\G_1, \G_5] = 1.$

The last proposition deals with the relations between $\G_{i'}$
and $\G_i$ in  $\pi_1 (\mathbb{C}^2 - S_2)$, where\\ $i =
9,13,15,20$.\end{proof}

\begin{prs} \label{prs4_6}
The following relations in $\pi_1 (\mathbb{C}^2 - S_2)$ hold:
$$
\begin{array}{ll}
\mbox{\emph{(i)}}  \quad \G_{13'} = (\G^2_{13})_{\G^2_{17}}
\Gmo_{13} &
\mbox{\emph{(ii)}} \quad \G_{9'} = (\G^2_9)_{\G^2_{2\!1}} \Gmo_9 \\
\mbox{\emph{(iii)}}  \quad \G_{15'} = \G_{15} & \mbox{\emph{(iv)}}
\quad \G_{20'} = \G_{20}
\end{array} $$
\end{prs}

\begin{proof}

(i) From the braid $\wz_{13,13'}$ in $\vp_{2,3}$, we induce:
$$
\G_{13}  =  \Gmo_2 \Gmo_{2'} \Gmo_4 \Gmo_{4'} \Gmo_{13}
\Gmo_{17'}\G_{17} \G_{13'} \Gmo_{17} \Gmo_{17'}  \G_{13} \G_{4'}
\G_4 \G_{2'} \G_2.
$$
Using $[\bG_4, \bG_2] = [\bG_4, \bG_{13}] = 1$ and $\G_2 =
\G_{2'}, \G_{17} = \G_{17'}$, we get
$$
1 =  \Gmt_2 \Gmo_{13} \G^2_{17}\G_{13'} \Gmt_{17} \G_{13} \G^2_{2}
\Gmo_{13} \quad \mbox{or}:
$$
$$
1 = \Gmt_{2}  (\G^2_{17})_{\G_{13}} \cdot (\Gmo_{13}
\G_{13'})\Gmt_{17} (\G^2_{2})_{\Gmo_{13}}.
$$
By $[\bG_2, \bG_{17}]=1$, we get
$$
\Gmo_{13} \G_{13'} = (\Gmt_{17})_{\G_{13}} (\Gmt_2)_{\Gmo_{13}}
\G^2_2 \G^2_{17} =  (\Gmt_{17})_{\G_{13}}
((\Gmt_{17})_{\Gmo_{13}})^{-1} \underbrace{(\Gmt_{17})_{\Gmo_{13}}
(\Gmt_2)_{\Gmo_{13}}\G^2_2 \G^2_{17} }_{F_{13}}.
$$
By Proposition (\ref{prs4_5}), the braids in $F_{13}$: $\G_2,
\G_{17}, (\G_2)_{\Gmo_{13}} , (\G_{17})_{\Gmo_{13}}$ are images of
a good quadrangle by $\a_4$ in $\widetilde{B}_5$, and thus $F_{13}
= 1$ (by Lemma (\ref{lem2_1}) on good quadrangles in
$\widetilde{B}_n$). Thus
$$
\Gmo_{13} \G_{13'} = \Gmo_{13} \Gmt_{17} \G^2_{13} \G^2_{17}
\Gmo_{13}
$$
or
$$ \G_{13'}  = (\Gmt_{13})_{\Gmt_{17}} \cdot \Gmo_{13}.
$$

(ii) We apply the same procedure as in (i) to the braid
$\tilde{Z}_{9,9'}$ from $\varphi_{2,3}$.

(iii)  Taking the complex conjugate of $\wz_{15,15'}$ in
$\vp_{2,1}$, we induce the relation (using $\G_1 = \G_{1'},
\G_{11} = \G_{11'})$:
$$
\G_{15'} = \Gmt_1 \Gmt_{11} \G_{15} \G^2_{11} \G^2_1
$$
or
$$
1 = \G^2_{11} \G^2_1 (\G_{15'} \Gmo_{15})(\Gmt_1)_{\Gmo_{15}}
(\Gmt_{11})_{\Gmo_{15}}
$$
$$
\G_{15} \cdot \Gmo_{15'} = (\Gmt_1)_{\Gmo_{15}}
(\G^2_{11})_{\Gmo_{15}} \G^2_{11} \G^2_{1}.
$$
By the same method as in (i) (using $\a_1 : \widetilde{B}_5
\rightarrow G_{2,1})$, we get that
$$
\G_{15} \Gmo_{15'} = 1 \quad \mbox{or} \quad \G_{15} = \G_{15'}.
$$
(iv) Taking the complex conjugation of $\wz_{20,20'}$ in
$\vp_{2,2}$, we induce the relation (using $\G_2 = \G_{2'}, \G_{8}
= \G_{8'})$, and $[\bG_{12}, \bG_{20}] = 1)$:
$$
\G_{20'} = \Gmt_2 \Gmt_8 \G_{20} \G^2_8 \G^2_2
$$
and we proceed as in (iii).  Thus: $\G_{20} = \G_{20'}$.
\end{proof}

These propositions show that $\pi_1 (\mathbb{C}^2 - S_2)$ is
generated only by $\{\G_i \}^{24}_{i=1}$, since the $\{\G_{i'}
\}^{24}_{i=1}$ can be expressed only in terms of the
$(\G_{i'})$'s. Our last goal is to prove the following:
\begin{thm} \label{thm4_1}
$G_2 = \po(\C-S_2)$ is a quotient of $\wb_{16}$.
\end{thm}

\begin{proof} We need to build an epimorphism
$\tilde{\a}:\wb_{16} \rightarrow G_2$. But first we build a new
representation for $B_{16}$. Consider the geometric model ($D,K$),
$\#K=16$ as in figure 22. Let $\{t_i\}_{i\in\, I}, \linebreak I =
\{1 \leq i \leq 24, i \neq 1,3,5,8,11,12,16,17,22,  i \in \Z\}$
segments that connect points in $K$ and $T_i$ be the half-twists
corresponding to $t_i$ (that is, $T_i = H(t_i),\, i \in I$).

\begin{center}
\epsfig{file=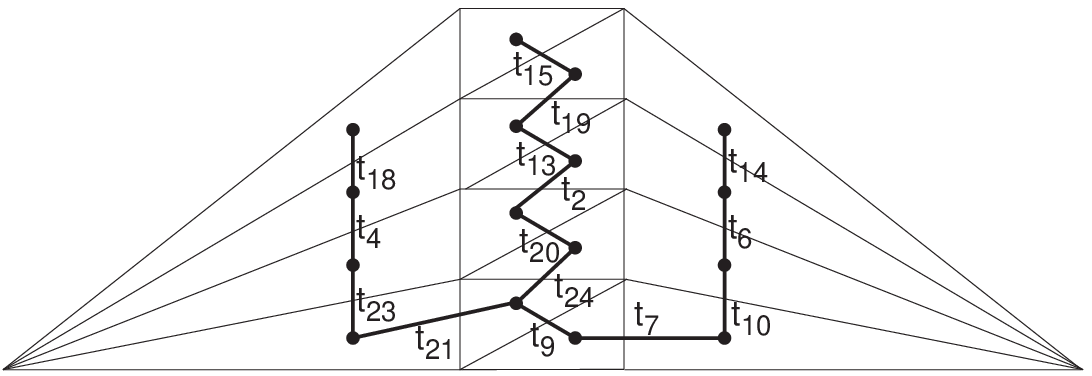}\\
\small{Figure 22}
\end{center}

\begin{lemma} \label{lem4_3}
There exists a presentation of $B_{16} $ when the generators are
$\{T_i\, | \, i\in I\}$ and the relations are:\\ \indent $\langle
T_i, T_j\rangle = 1$ \quad {if}\quad $T_i,T_j$ are consecutive,
\\ \indent $[ T_i, T_j] = 1$ \quad {if}\quad $T_i,T_j$ are
disjoint, \\ \indent $[T_9,T_{24}T_{21}T^{-1}_{24}]=1$.
\end{lemma}

\begin{proof} This is a standard consequence of the
usual presentation of $B_{16}(D,K)$ (see \cite{MoTe1}).\end{proof}

Let $X,Y \in B_{16}$  be transversal half-twists and let $\wb_{16}
= B_{16}/\langle [X,Y]\rangle$. By the previous lemma, $\wb_{16}$
is generated by $\{\tilde{T}_i\}_{i\in I}$ (where $\tilde{T}_i$
are the images of $T_i$ in $\wb_{16}$), and has same relations.

Using Remark \ref{rem4_3} we can define the ``missing" $T_j$'s
(where $1 \leq j \leq 24, j \notin I$). We begin with $j =
8,11,12,16,17,22:$\\
$$T_8 = (T_{10})_{T^{-1}_7 T^{-1}_9 T_{24}} \quad  T_{12} =
(T_{6})_{T^{-1}_8 T^{-1}_{20} T_{2}} \quad
T_{11} = (T_{14})_{T_{12} T_{13} T^{-1}_{19}}$$ (we can use $T_8$
and $T_{12}$ since these $T$'s are already defined)
$$T_{22} = (T_{24})_{T_{20} T_{21} T_{23}} \quad T_{17} = (T_{2})_{T_{22}
T^{-1}_{3} T_{4}} \quad
T_{16} = (T_{17})_{T^{-1}_{19} T_{15} T_{18}} $$ (we used
\,$\G_{20}= \G_{20'}$).

In order to find out how to define $T_3$ (and $T_5$), we look at a
relation induced from $\vp_{2,5}$ ($\vp_{2,6}$). Looking at the
braid $(Z_{3,4'})_{Z^2_{4',23}Z^2_{4',18}}$ from $\vp_{2,5}$, we
get the relation $\G_3 = (\G_4)_{\G^{-1}_{18}\G^{-1}_{23}}$. Thus
we define $T_3 = (T_4)_{T^{-1}_{18}T^{-1}_{23}}$. In the same way
we define $T_5 = (T_6)_{T^{-1}_{10}T^{-1}_{14}}$. By Remark
\ref{rem4_3}, let $T_1 =
(T_5)_{T^{-1}_{7}T^{-1}_{11}T^{-1}_{15}}$.

Denoting by $\{\tilde{T}_j\}_{j=1}^{24}$ the images of
$\{{T}_j\}_{j=1}^{24}$ in $\wb_{16}$, we can say that $\wb_{16}$
is generated by $\{\tilde{T}_j\}_{j=1}^{24}$ with the same
relations as above, and when the $\{\tilde{T}_j\}_{j \notin I}$
are defined as above. Define $\tilde{\a}(\tilde{T_j}) = \G_j,\, 1
\leq j \leq 24$. By Remark \ref{rem4_2}, it is easy to see that
$\forall i,j$ such that $T_i$ and  $T_j$ are consecutive, $\langle
\G_i, \G_j\rangle = 1$; and when $T_i$ and  $T_j$ are disjoint, $[
T_i, T_j] = 1$. The relations induced from the action of taking
quotient by $\langle [X,Y]\rangle$ (when $X,Y \in B_{16}$ are
transversal) are also preserved, due to Propositions \ref{prs4_4}
and \ref{prs4_5}. Also, $\tilde{\a}$ is an epimorphism, since for
every generator $\G_j$ of $G_2$ there exists a $\tilde{T}_j$ s.t.
$\tilde{\a}(\tilde{T}) = \G_j$. Thus $G_2 \simeq \wb_{16} /
\mbox{Ker}\tilde{\a}$.\end{proof}

\subsection[Computation for $X_1$]{Computation for $X_1$} \label{sec4_2}

As in subection (\ref{sec4_1}), we can compute the local relations
induced from each local braid monodromy.  However, a quotient of
the fundamental group of $X_1$ -- called the stabilized
fundamental group -- was already computed in \cite{ADKY}. Noticing
that $X_1$ can be regarded as a double cover of
$\mathbb{C}\mathbb{P}^1 \times \mathbb{C}\mathbb{P}^1$ branched
along a smooth algebraic curve of degree (4,4), we can use
\cite[Theorem 4.6]{ADKY}.

Let $\Theta_1 : \pi_1 (\mathbb{C}^2 - S_1) \rightarrow S_n$ be the
geometric monodromy representation morphism (here $n=16$).

\begin{thm}:  Let $K_1$ be the normal subgroup of $\pi_1 (\mathbb{C}^2 -
S_1)$ generated by all commutators $[\g_1, \g_2], \g_1, \g_2$ --
geometric generators of $\pi_1 (\mathbb{C}^2 - S_1)$, such that
$\Theta_1(\g_1)$ and $\Theta_1(\g_2)$ are disjoint transpositions.  \\
Let
$$
G^0_1 \doteq (\ker(\Theta_1: \pi_1 (\mathbb{C}^2 - S_1)
\rightarrow S_n) \bigcap \ker (\emph{deg:\,}  \pi_1 (\mathbb{C}^2
- S_1) \rightarrow \mathbb{Z}))/K_1 \ ,
$$
where \emph{deg} is the degree morphism. Then
$$
{\rm Ab} (G^0_1) \simeq (\mathbb{Z} \oplus \mathbb{Z}_2)^{15}
\quad \mbox{and} \   \quad [G^0_1, G^0_1 ] \simeq \mathbb{Z}_2
\oplus \mathbb{Z}_2
$$
\end{thm}
\begin{proof} See \cite[Thm. 4.6]{ADKY}.\end{proof}

The group $\pi_1 (\mathbb{C}^2 - S_1)/K_1$ is called the
\emph{stabilized fundamental group}.  Note that the stabilization
procedure does not affect $\pi_1 (\mathbb{C}^2 - S_2) = G_2$,
since $G_2 \simeq \wb_{16} / \mbox{Ker}\tilde{\a}$, and $\wb_{16}$
is already stabilized.

\subsection[Comparing the BMT's]{Comparing the BMT's}

In this subsection we prove that the BMF's of the  branch curves
$X_1$ and $X_2$ are not equivalent.  We will do this by looking at
the stabilized fundamental groups, related to $X_1$ and $X_2$. We
denote by $K_2$ the normal subgroup of $\pi_1 (\mathbb{C}^2 -
S_2)$ generated by all commutators $[\g_1, \g_2], \g_1, \g_2$ -
geometric generators of $\pi_1 (\mathbb{C}^2 - S_2)$, such that
$\Theta_2(\g_1)$ and $\Theta_2(\g_2)$ are disjoint transpositions
(here $\Theta_2 : \pi_1 (\mathbb{C}^2 - S_2) \rightarrow S_{16}$
is the geometric monodromy  morphism). As was noted, $\pi_1
(\mathbb{C}^2 - S_2) / K_2 = \pi_1 (\mathbb{C}^2 - S_2)$. We also
note that $K_1 = K_2$; since it is enough to pick one pair of
geometric generators (e.g., $x_2$ and $(x_2)_{x_3x_1}$, when the
$x_i$'s are geometric generators), and define $K_1 = K_2 = \langle
[x_2,(x_2)_{x_3x_1}]\rangle$.
\begin{thm}
$G_1/K_1 \not\simeq G_2/K_2$.
\end{thm}

\begin{proof} Note that $\wb_{16} /
\mbox{Ker}\tilde{\a} \simeq G_2 \simeq G_2/K_2 $. Denote
$$G^0_2 \doteq (\ker(\Theta_2: \pi_1 (\mathbb{C}^2 - S_2) \rightarrow
S_n) \bigcap \ker (\mbox{deg:\,}  \pi_1 (\mathbb{C}^2 - S_2)
\rightarrow \mathbb{Z}))/K_2.
$$
It is known from \cite{RobbT} what is the commutant subgroup of
$$\tilde{P}_{n,0} = \mbox{ker} (\wb_n \rightarrow S_n) \bigcap \mbox{ker}
(\mbox{deg} : \wb \rightarrow \Z);$$ Explicitely,
$[\tilde{P}_{n,0},\tilde{P}_{n,0}]$ is isomorphic to $\Z_2$.
Therefore, $[G^0_2,G^0_2]$ is a subgroup of $\Z_2$, whereas
$[G^0_1,G^0_1] \simeq \Z_2\oplus \Z_2$. But if $G_1/K_1 \simeq
G_2/K_2$ were isomorphic, then these two commutant subgroups would
be equal.\end{proof}

\begin{remark} \emph{We believe that an explicit computation of $\pi_1
(\mathbb{C}^2 - S_1)$ (as in \cite{Mo}, \cite{MoTe5}) would have
shown that $K_1 = \{e\}$.}
\end{remark}

\subsection[Computation for the Galois
covers]{Computation for the Galois covers}

Let $\widetilde{\pi}_i : \widetilde{X}_i \rightarrow \mathbb{C}^2
$ be the Galois covering corresponding to $\pi_1$ (see \cite{Mo}
for definitions).  Recall that $\pi_1(\widetilde{X}_i) = \ker
\Theta_i/\langle\G^2_{i,j}\rangle$ where  $ \Theta_i : \pi_1
(\mathbb{C}^2 - S_i, \star) \rightarrow S_n, \quad n = \mbox{deg}
\pi_i \ (i = 1 \ \mbox{or} \ 2;$ the degree is the same) and
$\{\G_{i,j} \}$ are the generators of $\pi_1 (\mathbb{C}^2 -
S_i)$, for $i = 1,2$.

In \cite{ATCM} it was proved that $\pi_1 (\widetilde{X}^{Aff}_1) =
\{e\}$.  This is also the case for $X_2$. We know that the
divisibility index of (the embedding of) $X_2$ is 1. Since $G_2$
is a quotient of $\wb_{16}$, we can now use  \cite[Theorem 4.1]{L}
to prove that $\pi_1 (\widetilde{X}^{Aff}_2) = \{e\}$.\\

\noindent \underline{{\bf The Main Result}}: Since the stabilized
fundamental groups induced from them are not isomorphic, $\vp_1$
is not Hurwitz-equivalent to $\vp_2$. Therefore, $X_1$ and $X_2$
are not BMT--equivalent. Note that this inequivalence cannot be
deduced from the computation of the fundamental groups of the
Galois covers, as these groups are isomorphic.

\section[Appendix]{Appendix: Invariance rules for the BMF of a
5-point}\label{sec5}

This appendix shows that the BMF of a 5-point is invariant under
certain braids. We focus on the BMF $\vp_{2,3}$, where the
invariance rules for the other $\vp_{i,j}$ ($i=1,2,\,1\leq j \leq
10$) are calculated in the same way.

Recall that two factorizations are Hurwitz equivalent if they are
obtained from
each other by a finite sequence of Hurwitz moves.\\
\textbf{Definition}:\, \underbar{A factorized expression invariant
under $h$}

Let $t=t_1\cdot\ldots\cdot t_m$ be a factorized expression in a
group $G$. We say that $t$ is invariant under $h \in G$ if
$(t_1)_h\cdot\ldots\cdot (t_m)_h$ is Hurwitz equivalent to
$t_1\cdot\ldots\cdot t_m$.\\

We recall now a few invariance rules (see \cite[section 3]{MoTe4}):\\
\emph{Invariance rule} II: $Z^2_{i,j\,j'}$ ($Z^2_{i\,i',j\,j'}$)is
invariant under $Z^q_{j\,j'}$ (resp. $Z^q_{j\,j'}Z^p_{i\,i'}$).\\
\emph{Invariance rule} III: $Z^{(3)}_{i,j\,j'}$ is invariant under
$Z^q_{j\,j'}$.

For our purposes (see the last paragraph in the proof of Lemma
(\ref{lem3_2})), it is enough to prove the following
\begin{prs} \label{prs5_1} $\vp_{2,3}$ is invariant under
$(Z_{1\,1'}Z_{3\,3'})^p(Z_{2\!1\,2\!1'} Z_{16\,16'})^q\,\, \forall
p,q\in \Z$.
\end{prs}

\begin{proof} We first look at the factors outside
$(F_3\cdot (F_3)_{\vartheta})_{Z_{\alpha_3}}$. By the Invariance
rule II, the factors
$Z^2_{3\,3',9},Z^2_{9',16\,16'},\tilde{Z}^2_{9',16\,16'},\tilde{Z}^2_{3\,3',9}$
are invariant under $Z_{3\,3'}$ and $Z_{16\,16'}$; by
(\cite[invariance remark (iv)]{MoTe4}), these factors are also
invariant under $Z_{1\,1'}$ and $Z_{2\!1\,2\!1'}$ (since the paths
are disjoint). Again, by the same invariance remark,
$\tilde{Z}_{9,9'}$ is invariant under $Z_{i\,i'}\,\,i=1,3,16,21$.
By the Invariance rule III, the factors $Z^{(3)}_{1\,1',9}$ and
$\bar{Z}^{(3)}_{9`,2\!1\,2\!1'}$ are invariant under $Z_{1\,1'}$
and $Z_{2\!1\,2\!1'}$ (and also under $Z_{3\,3'}$ and
$Z_{16\,16'}$ by the Invariance remark (iv)).

We note that the conjugation by the braid $Z_{\alpha_3}$ is
actually conjugation by $Z^2_{3\,3',9}Z^2_{1\,1',9}$, so it is
also invariant under $Z_{i\,i'}\,\,i=1,3,16,21$ (by invariance
rule II and remark (iv)). When looking at the expression $F_3\cdot
(F_3)_{\vartheta}$, we see that this case was already done in
\cite[invariance property 8.7]{AT}; it was proved there that
$F_3\cdot (F_3)_{\vartheta}$ is invariant under
$(Z_{1\,1'}Z_{3\,3'})^p(Z_{2\!1\,2\!1'}
Z_{16\,16'})^q$.\end{proof}


\begin{thebibliography}{99}


\bibitem{ATCM}
M. Amram, C. Ciliberto, R. Miranda, M. Teicher, {\em Braid
monodromy factorization for a non-prime $K3$ surface branch curve}
, Israel Journal of Mathematics, to appear.


\bibitem{AFT}M. Amram, M. Friedman, M. Teicher,\emph{ The
fundamental group of complement of a branch curve of a Hirzebruch
surface $F_{2,(2,2)}$}, submitted to Topology.

\bibitem{AT}
M. Amram, M. Teicher, \emph{The fundamental group of the
complement of the branch curve of the double torus}, Journal of
Mathematics, 40(4), (2003), 587-893.

\bibitem{ADKY}
D. Auroux, S. K. Donaldson, L. Katzarkov,  M. Yotov, {\em
Fundamental groups of complements of plane curves and symplectic
invariants}, Topology 43, (2004), 1285-1318.



\bibitem{CM}
C. Ciliberto, R. Miranda, {\em On the Gaussian map for canonical
curves of low genus}, Duke Mathemtical J., 61, No. 2 (1990),
417-442.

\bibitem{CMT}
C. Ciliberto, R. Miranda, M. Teicher, {\em Pillow degenerations of
K3 surfaces}, In: ``Applications of Algebraic Geometry to Coding
Theory, Physics, and Computation", NATO Science Series II, Vol. 36
(2001), 53-63.


\bibitem{FT}
M. Friedman, M. Teicher, {\em On the fundamental group related to
the Hirzebruch surface $F_1$}, submitted to Sci. China ser. A.




\bibitem{KuTe}V. S. Kulikov and M. Teicher, {\em Braid monodromy
factorizations and diffeomorphism types}, Izv. Ross. Akad. Nauk
Ser. Mat. 64(2), (2000), 89-120 , [Russian]; English transl.,
Izvestiya Math. 64(2), (2000), 311-341.

\bibitem{LT}
E. Liberman, M. Teicher, {\em The Hurwitz equivalence problem is
undecidable}, math.LO/0511153, preprint.

\bibitem{L}
C. Liedtke, {\em On Fundamental Groups of Galois closures of
generic projections}, Bonner Mathematische Schriften Nr. 367
(2004).


\bibitem{Mo} B. Moishezon,  {\em On cuspidal branch curves}, J. Algebraic
Geometry 2 (1993) no. 2, 309-384.

\bibitem{MoRoTe}B. Moishezon, A. Robb and M. Teicher, {\em On Galois covers
of Hirzebruch surfaces}, Math. Ann. 305, (1996), 493-539.

\bibitem{MoTe0}
B. Moishezon, M. Teicher, {\em Simply connected algebraic surfaces
of positive index}, Invent. Math. 89 (1987), 601-643.

\bibitem{MoTe1} B. Moishezon, M. Teicher, {\em Braid group technique
in complex geometry, I, Line arrangements in $\CP^2$}, Contemp.
Math. 78, (1988), 425-555.

\bibitem{MoTe2}B. Moishezon, M. Teicher, {\em Braid group technique in
complex geometry, II, From arrangements of lines and conics to
cuspidal curves}, Algebraic Geometry, Lecture Notes in Math., vol.
1479 (1990), 131-180.

\bibitem{MoTe4}B. Moishezon, M. Teicher, {\em Braid group techniques in
complex geometry IV: Braid monodromy of the branch curve $S_3$ of
$V_3 \rightarrow \CP^2$ and application to $\p:(\CP^2 - S_3,
\ast)$}, Contemp. Math. 162, (1993), 332-358.

\bibitem{MoTe5}B. Moishezon, M. Teicher, {\em Braid group techniques in
complex geometry, V: The fundamental group of complements of a
branch curve of Veronese generic projection}, Communications in
Analysis and Geometry 4, (1996), no. 1, 1-120.


\bibitem{RobbT}  A. Robb, {\em The topology of branch curves of complete
intersections}, Doctoral Thesis, Columbia University, (1994).


\bibitem{Robb}  A. Robb, {\em On branch curves of Algebraic Surfaces}, Stud.
Adv. Math. 5, (1997), 193-221.


\bibitem{VK}E.R. Van Kampen,{\em On the fundamental group of an algebraic
curve}, Amer. J. Math. 55 (1933), 255-260.

\end{thebibliography}
\end{document}